\documentclass[12pt,reqno]{amsart}
\usepackage{amsmath,amsfonts,amsthm,amsopn,amssymb,extarrows}
\usepackage{cite,marginnote}
\usepackage{bm}
\usepackage{color,enumitem,graphicx}
\usepackage[colorlinks=true,urlcolor=blue,
citecolor=red,linkcolor=blue,linktocpage,pdfpagelabels,
bookmarksnumbered,bookmarksopen]{hyperref}
\usepackage[english]{babel}

\usepackage[left=2.9cm,right=2.9cm,top=2.8cm,bottom=2.8cm]{geometry}
\usepackage[hyperpageref]{backref}

\numberwithin{equation}{section}

\makeindex
\makeindex

\newtheorem{theorem}{Theorem}[section]
\newtheorem{definition}[theorem]{Definition}
\newtheorem{lemma}[theorem]{Lemma}
\newtheorem{corollary}[theorem]{Corollary}
\newtheorem{proposition}[theorem]{Proposition}
\newtheorem{remark}[theorem]{Remark}

\newcommand{\bt}{\begin{theorem}}
	\newcommand{\et}{\end{theorem}}
\newcommand{\bl}{\begin{lemma}}
	\newcommand{\el}{\end{lemma}}
\newcommand{\bd}{\begin{definition}}
	\newcommand{\ed}{\end{definition}}
\newcommand{\bc}{\begin{corollary}}
	\newcommand{\ec}{\end{corollary}}
\newcommand{\bp}{\begin{proof}}
	\newcommand{\ep}{\end{proof}}
\newcommand{\bx}{\begin{example}}
	\newcommand{\ex}{\end{example}}
\newcommand{\bi}{\begin{exercise}}
	\newcommand{\ei}{\end{exercise}}
\newcommand{\bo}{\begin{proposition}}
	\newcommand{\eo}{\end{proposition}}
\newcommand{\br}{\begin{remark}}
	\newcommand{\er}{\end{remark}}
\newcommand{\beq}{\begin{equation}}
	\newcommand{\eeq}{\end{equation}}
\newcommand{\ba}{\begin{align}}
	\newcommand{\ea}{\end{align}}
\newcommand{\bn}{\begin{enumerate}}
	\newcommand{\en}{\end{enumerate}}
\newcommand{\bg}{\begin{align*}}
	\newcommand{\bcs}{\begin{cases}}
		\newcommand{\ecs}{\end{cases}}

	\newcommand{\bean}{\begin{eqnarray*}}
		\newcommand{\eean}{\end{eqnarray*}}

	\def\bd{\mathrm{bd}\,}

	\title[On nodal solutions for a Kirchhoff-type problem]{On nodal solutions with a prescribed number of nodes for a Kirchhoff-type problem}
	
	\author[H.~Fan]{Haining Fan}
    \author[M. Squassina]{Marco Squassina}
	\author[J.~Zhang]{Jianjun Zhang}

\address[H.~Fan]{\newline\indent School of Mathematics
		\newline\indent
		China University of Mining and Technology
		\newline\indent
		Xuzhou, Jiangsu, 221116, P. R. China}
	 \email{\href{mailto:fanhaining888@163.com}{fanhaining888@163.com}}

\address[M. Squassina]{\newline\indent Dipartimento di Matematica e Fisica
\newline\indent
Universit\`a Cattolica del Sacro Cuore
\newline\indent
Via dei Musei 41, Brescia, Italy}\email{\href{marco.squassina@unicatt.it}{marco.squassina@unicatt.it}}

\address[J.~Zhang]{\newline\indent College of Mathematics and Statistics
\newline\indent
Chongqing Jiaotong University
\newline\indent
Xuefu, Nan'an, 400074, Chongqing, P. R. China}
\email{\href{mailto:zhangjianjun09@tsinghua.org.cn}{zhangjianjun09@tsinghua.org.cn}}

\date{\today}

\begin{document}
		
\begin{abstract}
We are concerned with the existence and asymptotic behavior of multiple radial sign-changing solutions with the nodal characterization  for a Kirchhoff-type problem involving the nonlinearity $|u|^{p-2}u(2<p<4)$ in  $\mathbb{R}^3$. By developing some useful analysis techniques and introducing a novel definition of the Nehari manifold for the auxiliary system of the equations, we show that,  for any positive integer $k$, the problem has a sign-changing solution $u_k^b$ changing signs exactly $k$ times. Furthermore, the energy of $u_k^b$ is strictly increasing in $k$, as well as some asymptotic behaviors of $u_k^b$ are obtained. Our result is a complement of [Deng Y, Peng S, Shuai W, {\it J. Funct. Anal.}, {\bf269}(2015), 3500-3527], where the case $2<p<4$ is left open.
\end{abstract}

\maketitle
	
\section{Introduction and main result}
\subsection{Background and motivation} In this paper, we discuss about the existence of infinitely many sign-changing solutions with the nodal characterization for the following Kirchhoff-type problem:
\begin{equation}\label{s1.1}
-\left(a+b\displaystyle\int_{\mathbb{R}^3}|\nabla u|^2{\rm d}x\right)\Delta u+V(x)u=|u|^{p-2}u,~ x\in\mathbb{R}^3,
\end{equation}
where $a$ is a positive constant,  $b$ is a small positive parameter, and $2<p<4$. Moreover, $V(x)$ is a continuous weight potential.

Problem (\ref{s1.1}) has a strong physical background. In fact, as a special case, the following Dirichlet problem
\begin{equation}\label{s1.02}
\left\{\begin{array}{ll}
-\left(a+b\displaystyle\int_{\Omega}|\nabla u|^2{\rm d}x\right)\Delta u+V(|x|)u=f(x,u),~ x\in\Omega,\\
u=0,~x\in\partial\Omega,
\end{array}\right.
\end{equation}
is closely related to the stationary analogue of the equation
\begin{equation*}
\rho\frac{\partial^2u}{\partial t^2}-\left(\frac{P_0}{h}+\frac{E}{2L}\displaystyle\int_0^L|\frac{\partial u}{\partial x}|^2{\rm d}x\right)\frac{\partial^2u}{\partial x^2}=0,
\end{equation*}
which is proposed by Kirchhoff in \cite{1} to describe the transversal oscillations of a stretched string. Kirchhoff's model takes into the changes in string length produced by vibration, so
the nonlocal term appears. Moreover, such nonlocal problems also appear in other fields as biological systems with $u$ is used to describe a process that depends on its own average. We refer
the readers to \cite{2,3} and the references therein for more physical background on Kirchhoff-type problems.

Mathematically, (\ref{s1.1}) or (\ref{s1.02}) is a nonlocal problem due to the presence of the nonlocal term $\displaystyle\int_{\mathbb{R}^3}|\nabla u|^2{\rm d}x\Delta u$ implies that
it is no longer a point-wise identity. This phenomenon brings some mathematical difficulties, and at the same time, makes the study of such a problem particularly interesting.
Actually, after Lions established an abstract functional analysis framework in \cite{4},   Kirchhoff-type problems have received much attention. By using the variational methods, a number of results on the existence and multiplicity of solutions for  Kirchhoff-type problems similar to
(\ref{s1.1}) or (\ref{s1.02}) have been obtained in the literature, such as positive solutions, sign-changing solutions, multiple solutions, ground states and semiclassical states,
see for example \cite{5,6,7,8,9,10,11,12,13,14} and the references therein. In particular, infinitely many  sign-changing solutions are established in \cite{15,16,17}, but no information
of nodal characterization was given. Besides, with the aid of the Non-Nehari manifold method and Nehari manifold method respectively, Tang-Cheng  \cite{18} and Ye  \cite{19}
both obtained the existence of the least energy nodal solutions with precisely two nodal domains under different conditions of $f(x,u)$ respectively.

However, regarding the multiple nodal solutions with the nodal characterization for  Kirchhoff-type problems, to the best of our knowledge, there are very few results in the context. When $b=0$, Bartsch-Willem \cite{20} and Cao-Zhu \cite{21}
considered the following semilinear equation
\begin{equation}\label{s1.03}
\left\{\begin{array}{ll}
-a\Delta u+V(|x|)u=f(|x|,u),~ x\in\mathbb{R}^N,\\
u\in H^1(\mathbb{R}^N),
\end{array}\right.
\end{equation}
they independently obtained infinitely many radial nodal solutions having a prescribed number of nodes. The case of the ball has also been studied by Struwe \cite{22}.
They first get the solution of (\ref{s1.03}) in each annulus, and then glue them by matching the normal derivative at each junction point. When $b>0$,  (\ref{s1.1}) or (\ref{s1.02}) can not be solved separately
on each annulus due to the appearance of the nonlocal term. Therefore, the method in \cite{20,21,22} can not be applied directly. Deng et al.\cite{23} resolve it by regarding the problem as
a system of $(k+1)$ equations with $(k+1)$ unknown functions $u_i$, each $u_i$ is supported on only one annulus and vanishes at the complement of it. We should point out that their
method strongly depend the following Nehari-type monotonicity condition  of $f$:
\begin{description}
  \item[(Ne)] $\frac{f(r,t)}{|t|^3}$ is increasing on $t\in(-\infty,0)\cup(0,+\infty)$ for every $r>0$.
\end{description}
Recently, Guo et al.\cite{24} obtained a similar result by assuming the following condition of $f(|x|,u)=K(|x|)f(u)$ in (\ref{s1.02}) :
\begin{description}
  \item[(\^{N}e)] there exists $\theta\in(0,1)$ such that for all $t>0$ and $\tau\in \mathbb{R}\backslash\{0\}$,
  \begin{equation*}
K(|x|)\left[\frac{f(\tau)}{\tau^3}-\frac{f(t\tau)}{(t\tau)^3}\right]sign(1-t)+\theta V_0\frac{|1-t^{-2}|}{\tau^2}\geq0.
\end{equation*}
\end{description}
Unfortunately, the assumptions (Ne) and (\^{N}e) both rule out some important cases such as $f(|x|,u)=|u|^{p-2}u$ for $p\in(2,4)$. In this paper, we aim to resolve this case.

\subsection{Main assumptions and results}
The first purpose of this paper is to find infinitely many radial sign-changing solutions to (\ref{s1.1}) with $2<p<4$. Assume that $V(x)$ satisfies the following condition.
\begin{description}
  \item[(V)] $V(x)$ is radial i.e. $V(x)=V(|x|)$, and $\displaystyle\inf_{x\in\mathbb{R}^3}V(x):=V_0>0$.
\end{description}
Then the first main result of this paper is stated as follows.
\begin{theorem}(Existence of nodal solutioins)\label{t1.1}
Assume that the assumption (V) holds and $3<p<4$. Then for every integer $k>0$, there exists $b^*>0$ which is defined in Lemma \ref{l2.4}  such that for any $b\in(0,b^*)$,  (\ref{s1.1}) admits a radial solution $u_k$, which changes exactly $k$-times.
\end{theorem}
\begin{remark}\label{r1.1}
It is easy to see that if $u$ is a solution of  (\ref{s1.1}), then $-u$ follows. Moreover, we should point that, if the domain of (\ref{s1.1}) is replaced by $B_R(0)\subset\mathbb{R}^3$, Theorem \ref{t1.1} also holds under the condition $2<p<4$ instead of $3<p<4$. Actually, we just need to replace the term $(ii)$ in Lemma 3.1 by $(ii')$: If $r_k\rightarrow R$, then $\varphi(\overrightarrow{\textbf{r}_k})\rightarrow+\infty$. Fortunately, the proof of it is analogous to the term $(i)$ in Lemma 3.1, and just need the assumption $2<p<4$. The remain proof is similar and only needs a little modification.
\end{remark}

The next aim of this paper is to present that the energy of $u_k$ is strictly increasing in $k$. Before state it clearly, we at first give some definitions. Define the Sobolev space
 \begin{equation*}
\mathcal{H}:=\left\{u\in H_r^1(\mathbb{R}^3)|\displaystyle\int_{\mathbb{R}^3}(a|\nabla u|^2+V(|x|)u^2){\rm d}x<+\infty\right\},
\end{equation*}
with the norm
 \begin{equation*}
\|u\|_{\mathcal{H}}:=\left(\displaystyle\int_{\mathbb{R}^3}(a|\nabla u|^2+V(|x|)u^2){\rm d}x\right)^\frac{1}{2}.
\end{equation*}
Since $a$ is a positive constant, we assume $a=1$ from then on.
Noticing the condition (V), we know that the embedding $\mathcal{H}\hookrightarrow H_r^1(\mathbb{R}^3)$ is continuous. Moreover, we have $\mathcal{H}\hookrightarrow L^q(\mathbb{R}^3)$, $q\in[2,6]$, and this embedding is compact for $q\in(2,6)$ by Strauss \cite{25}.
Then we can denote
\begin{equation*}
S_q:=\displaystyle\inf_{u\in\mathcal{H}\backslash\{0\}}\frac{\|u\|_{\mathcal{H}}^2}{|u|^2_q},
\end{equation*}
for $q\in(2,6]$.

Define the energy functional $I_b$ associated with (\ref{s1.1}) on $\mathcal{H}$ by
\begin{equation}\label{s1.2}
I_b(u):=\frac{1}{2}\displaystyle\int_{\mathbb{R}^3}(|\nabla u|^2+V(|x|)u^2){\rm d}x+\frac{b}{4}\left(\displaystyle\int_{\mathbb{R}^3}|\nabla u|^2{\rm d}x\right)^2-\frac{1}{p}\displaystyle\int_{\mathbb{R}^3}|u|^p{\rm d}x.
\end{equation}
Then, it is standard to show that $I_b\in C^2(\mathcal{H},\mathbb{R})$ and
\begin{align}\label{s1.0}
&~~~~\langle I'_b(u),v\rangle\notag
\\&=\displaystyle\int_{\mathbb{R}^3}(\nabla u\nabla v+V(|x|)uv){\rm d}x+b\displaystyle\int_{\mathbb{R}^3}|\nabla u|^2{\rm d}x\displaystyle\int_{\mathbb{R}^3}\nabla u\nabla v{\rm d}x-\displaystyle\int_{\mathbb{R}^3}|u|^{p-2}uv{\rm d}x,
\end{align}
for any $u,v\in\mathcal{H}$. Clearly, $u\in \mathcal{H}$ is a solution of (\ref{s1.1}) if and only if $u$ is a critical point of $I_b$. Furthermore, we show the second main result of this paper in the following.
\begin{theorem}(Monotonicity of energy)\label{t1.2}
Assume the assumptions of Theorem \ref{t1.1} hold, then the energy of $u_k$ is strictly increasing in $k$, i.e.
\begin{equation*}
I_b(u_{k+1})>I_b(u_k)~\text{and}~I_b(u_k)>(k+1)I_b(u_0)
\end{equation*}
for any integer $k\geq0$. Moreover, if $V(x)\in C^1(\mathbb{R}^3,\mathbb{R})$ and $\langle\nabla V(x),x\rangle\leq0$ for any $x\in \mathbb{R}^3$, $u_0$ is  a ground state radial solution of  (\ref{s1.1}).
\end{theorem}

We denote $u^b_k$ as the solution obtained in Theorem \ref{t1.1} to emphasize that it depends on $b$. Then our third main result shall show the convergence properties of it as $b\rightarrow0^+$ in the following theorem.
\begin{theorem}(Asymptotic behaviour)\label{t1.3}
Let $u^b_k$ be obtained in Theorem \ref{t1.1} with $b\rightarrow0^+$. Then, $u^b_k\rightarrow u^0_k$ in $\mathcal{H}$ up to a subsequence, where $u^0_k$ is a least energy radial solution which changes sign exactly $k$-times of the following equation:
\begin{equation}\label{s1.3}
-\Delta u+V(x)u=|u|^{p-2}u,~ x\in\mathbb{R}^3.
\end{equation}
\end{theorem}
\begin{remark}\label{r1.2}
We point out that, similarly as Remark \ref{r1.1},  for the domain of (\ref{s1.1}) is replaced by $B_R(0)\subset\mathbb{R}^3$, Theorems \ref{t1.2} and \ref{t1.3} also hold under the condition $2<p<4$ instead of $3<p<4$ except that we do not know whether that $u_0$ is a ground state solution  of  (\ref{s1.1}). We refer the readers to see the proof of Theorem \ref{t1.2} for details.
\end{remark}
\subsection{Main novelty and strategy} In \cite{23}, the authors adopt a gluing argument to obtain the existence of nodal solutions to problem \eqref{s1.1} with $4<p<6$ by investigating system (\ref{s2.1}) involving $(k+1)$ equations. One of the main ideas in \cite{23} is to introduce the following Nehari manifold
$$
N_k:=\left\{(u_1,\cdots,u_{k+1})\in\mathcal{\textbf{H}}_k|u_i\neq0,\langle\partial_{u_i}E_b(u_1,\cdots,u_{k+1}),u_i\rangle=0\right\},
$$
where $\textbf{H}_k$ and $E_b$ are given in Section 2. Compared to \cite{23}, the problem addressed in the present paper becomes more difficult due to $2<p<4$. With the presence of the nonlocal term $\displaystyle\int_{\mathbb{R}^3}|\nabla u|^2{\rm d}x\Delta u$, it seems tough to get the boundedness of the associated (PS) sequence in the case of $2<p<4$. To overcome this obstacle, we turn to consider a subset $N_k^-$ of the Nehari manifold $N_k$ by imposing the following additional constraint
$$
(4-p)\displaystyle\int_{B_i}|u_i|^p{\rm d}x<2\|u_i\|_i^2~\text{for}~i=1,\cdots,k+1.
$$
This constraint implies that $N_k^-$ only contains local maximum points of the associated fibering maps and then is natural. More important is that with this additional constraint, the associated (PS) sequence can be proved to be bounded. We should point out that this trick also provides an idea to study other types of Krichhoff problems with the nonlinearity $|u|^{p-2}u, 2<p<4$, by using variational methods.

The remainder part of the paper is the following. In Section 2, we first clarify an appropriate variational framework to solve (\ref{s1.1}), and then deal with an auxiliary system
whose functional is closely related to the original energy functional.
In Section 3, by using the critical point obtained in Section 2 as a building block, we obtain the ideal solution of (\ref{s1.1}), and complete the proof of Theorem \ref{t1.1}.
In Section 4, we study some asymptotic behaviors  of the solutions obtained in Section 2 with the help of some properties of the modified energy functional and some analysis techniques, and
prove Theorems \ref{t1.2} and \ref{t1.3}.

\subsection{Notation}
\begin{enumerate}
\item[$\bullet$]$\rightarrow$ (resp. $\rightharpoonup$) the strong (resp. weak) convergence.
\item[$\bullet$]$|\cdot|_q$ the usual norm of the space $L^q(\mathbb{R}^3),(1\leq q<\infty)$.
\item[$\bullet$]$|\cdot|_\infty$ denote the norm of the space $L^\infty(\mathbb{R}^3)$.
\item[$\bullet$]$C$ or $C_i(i=0,1,2,\ldots)$ denote some positive constants that may change from line to line.
\end{enumerate}
\section{\normalsize{Preliminary}}
\quad\quad
In this section, we establish the variational framework and study some properties of the energy functional corresponding to a system of $(k+1)$-equations associated with  (\ref{s1.1}). At first, we give some definitions and introduce the working space.

Fix $k\in\mathbb{N}$, denote
\begin{equation*}
\Lambda_k:=\left\{\textbf{r}_k=(r_1,\cdots,r_k)\in\mathbb{R}^k|0=:r_0<r_1<\cdots<r_k<r_{k+1}:=\infty\right\},
\end{equation*}
and
\begin{equation*}
B_i:=B_{i,\textbf{r}_k}=\{x\in\mathbb{R}^3|r_{i-1}<|x|<r_i\},
\end{equation*}
for $i=1,\cdots,k+1$. The idea of this paper is to obtain the solution of (\ref{s1.1}) in each $B_i$ first and then glue them by matching the normal derivative at each junction point. For an element $\textbf{r}_k\in \Lambda_k$ fixed and so a family of annuli $\{B_i\}_{i=1}^{k+1}$ also. Define the Sobolev space
\begin{equation*}
\mathcal{H}_i:=\{u\in H_0^1(B_i)|u(x)=u(|x|),~u(x)=0,~\text{if}~x\not\in B_i\},
\end{equation*}
equipped with the norm
\begin{equation*}
\|u\|_i:=\left(\displaystyle\int_{B_i}(|\nabla u|^2+V(|x|)u^2){\rm d}x\right)^\frac{1}{2},
\end{equation*}
for $i=1,\cdots,k+1$. Then we let $\mathcal{\textbf{H}}_k:=\mathcal{H}_1\times\cdots\times\mathcal{H}_{k+1}$ and set the functional $E_b:\mathcal{\textbf{H}}_k\rightarrow\mathbb{R}$ by
\begin{align}
E_b(u_1,\cdots,u_{k+1}):&=\frac{1}{2}\displaystyle\sum_{i=1}^{k+1}\|u_i\|_i^2+\frac{b}{4}\displaystyle\sum_{i=1}^{k+1}\left(\displaystyle\int_{B_i}|\nabla u_i|^2{\rm d}x\right)^2\notag
\\&+\frac{b}{4}\displaystyle\sum_{i\neq j}^{k+1}\left(\displaystyle\int_{B_i}|\nabla u_i|^2{\rm d}x\displaystyle\int_{B_i}|\nabla u_j|^2{\rm d}x\right)-\frac{1}{p}\displaystyle\sum_{i=1}^{k+1}\displaystyle\int_{B_i}|\nabla u_i|^p{\rm d}x,\notag
\end{align}
where $u_i\in\mathcal{H}_i$ for $i=1,\cdots,k+1$.

Standard calculations shows that $E_b(u_1,\cdots,u_{k+1})=I_b(\displaystyle\sum_{i=1}^{k+1}u_i)$. Also, each component $u_i$ of a critical point of $E_b$ satisfies
\begin{equation}\label{s2.1}
\left\{\begin{array}{ll}
-\left(1+b\displaystyle\sum_{j=1}^{k+1}\displaystyle\int_{B_j}|\nabla u_j|^2{\rm d}x\right)\Delta u_i+V(|x|)u_i=|u_i|^{p-2}u_i,~ x\in B_i,\\
u_i=0,~x\not\in B_i.
\end{array}\right.
\end{equation}
Since $2<p<4$, we define the $N_k^-$ corresponding to the local maximum points of the fibering map $\phi_{u_1,\ldots,u_{k+1}}(t_1,\cdots,t_{k+1})$ by
\begin{align}
N_k^-:&=\{(u_1,\cdots,u_{k+1})\in\mathcal{\textbf{H}}_k|u_i\neq0,\langle\partial_{u_i}E_b(u_1,\cdots,u_{k+1}),u_i\rangle=0,\notag
\\&~~~~~~(4-p)\displaystyle\int_{B_i}|u_i|^p{\rm d}x<2\|u_i\|_i^2~\text{for}~i=1,\cdots,k+1\}\notag
\\&=\{(u_1,\cdots,u_{k+1})\in\mathcal{\textbf{H}}_k|u_i\neq0,\|u_i\|_i^2+b\left(\displaystyle\int_{B_i}|\nabla u_i|^2{\rm d}x\right)^2\notag
\\&~~~~~~+b\displaystyle\int_{B_i}|\nabla u_i|^2{\rm d}x\displaystyle\sum_{j\neq i}^{k+1}\displaystyle\int_{B_j}|\nabla u_j|^2{\rm d}x=\displaystyle\int_{B_i}|u_i|^p{\rm d}x\notag
\\&~~~~~~(4-p)\displaystyle\int_{B_i}|u_i|^p{\rm d}x<2\|u_i\|_i^2~\text{for}~i=1,\cdots,k+1\},\notag
\end{align}
where $\phi_{u_1,\ldots,u_{k+1}}(t_1,\cdots,t_{k+1}):=E_b(t_1u_1,\cdots,t_{k+1}u_{k+1})$.

Let us check that $N_k^-$ is nonempty in $\mathcal{\textbf{H}}_k$.
\begin{lemma}\label{l2.1}
Assume that (V) hold and $(u_1,\cdots,u_{k+1})\in\mathcal{\textbf{H}}_k$ with $\frac{\left(\displaystyle\int_{B_i}|u_i|^p{\rm d}x\right)^\frac{2}{p}}{\|u_i\|_i^2}\geq(2S_p)^{-1}$ for $i=1,\ldots,k+1$. Then, for each $b\in(0,b_*)$ with
\begin{equation*}
b_*:=\min\left\{\frac{p-2}{4-p}\left(\frac{4-p}{2}\right)^\frac{2}{p-2}(2S_p)^\frac{-p}{p-2},\widehat{b}\right\},
\end{equation*}
and
\begin{equation*}
\widehat{b}:=\frac{p-2}{4-p}\left(1+k2^{\frac{2}{p-2}}\left(\frac{2}{4-p}\right)^\frac{2}{p-2}\right)^{-1}\left(\frac{4-p}{2}\right)^\frac{2}{p-2}(2S_p)^\frac{-p}{p-2}.
\end{equation*}
there is a unique $(k+1)$-tuple $(t_1,\cdots,t_{k+1})\in(\mathbb{R}_{>0})^{k+1}$ of positive numbers such that $(t_1u_1,\cdots,t_{k+1}u_{k+1})\in N_k^-$.
\end{lemma}
\begin{proof}
For a fixed $(u_1,\cdots,u_{k+1})\in\mathcal{\textbf{H}}_k$ with $u_i\neq0$, $(t_1u_1,\cdots,t_{k+1}u_{k+1})\in N_k^-$ if and only if
\begin{equation}\label{s2.2}
t_i^2\|u_i\|_i^2+t_i^4b\left(\displaystyle\int_{B_i}|\nabla u_i|^2{\rm d}x\right)^2+\mu bt_i^2\displaystyle\int_{B_i}|\nabla u_i|^2{\rm d}x\displaystyle\sum_{j\neq i}^{k+1}t_j^2\displaystyle\int_{B_j}|\nabla u_j|^2{\rm d}x-t_i^p\displaystyle\int_{B_i}|u_i|^p{\rm d}x=0
\end{equation}
and
\begin{equation}\label{s2.3}
(4-p)t_i^p\displaystyle\int_{B_i}|u_i|^p{\rm d}x<2t_i^2\|u_i\|_i^2,
\end{equation}
for each $i=1,\cdots,k+1$ and $\mu=1$.

Denote
\begin{equation}\label{s2.4}
Z:=\{\mu|0\leq\mu\leq1, ~\text{and}~ (\ref{s2.2})-(\ref{s2.3})~\text{are~uniquely~solvable~in}~(\mathbb{R}_{>0})^{k+1}\}.
\end{equation}
We shall show that $Z=[0,1]$ in the following three steps.

\textbf{Step}\textbf{1:} We will prove that $0\in Z$. Set
\begin{align}
g_i(t):&=t^4\left(t^{-2}\|u_i\|_i^2+b\left(\displaystyle\int_{B_i}|\nabla u_i|^2{\rm d}x\right)^2-t^{p-4}\displaystyle\int_{B_i}|u_i|^p{\rm d}x\right)\notag
\\&=t^4h_i(t),\label{s2.5}
\end{align}
we have
\begin{align}
h'_i(t):&=-2t^{-3}\|u_i\|_i^2-(p-4)t^{p-5}\displaystyle\int_{B_i}|u_i|^p{\rm d}x\notag
\\&=t^{-3}\left((4-p)t^{p-2}\displaystyle\int_{B_i}|u_i|^p{\rm d}x-2\|u_i\|_i^2\right).\label{s2.6}
\end{align}
Let
\begin{equation*}
T_{u_i}:=\left(\frac{2\|u_i\|_i^2}{(4-p)\displaystyle\int_{B_i}|u_i|^p{\rm d}x}\right)^\frac{1}{p-2},
\end{equation*}
we deduce from (\ref{s2.6}) that $h_i(t)$ is decreasing in $(0,T_{u_i})$ and increasing in $(T_{u_i},+\infty)$. Moreover,
\begin{align}
h_i(T_{u_i})&=b\left(\displaystyle\int_{B_i}|\nabla u_i|^2{\rm d}x\right)^2-\frac{p-2}{4-p}\left(\frac{4-p}{2}\right)^\frac{2}{p-2}\|u_i\|_i^2\left(\frac{\displaystyle\int_{B_i}|u_i|^p{\rm d}x}{\|u_i\|_i^2}\right)^\frac{2}{p-2}\notag
\\&<b\left(\displaystyle\int_{B_i}|\nabla u_i|^2{\rm d}x\right)^2-\frac{p-2}{4-p}\left(\frac{4-p}{2}\right)^\frac{2}{p-2}(2S_p)^{\frac{-p}{p-2}}\|u_i\|_i^4\notag
\\&<\left(b-\frac{p-2}{4-p}\left(\frac{4-p}{2}\right)^\frac{2}{p-2}(2S_p)^\frac{-p}{p-2}\right)\|u_i\|_i^4<0,\label{s2.7}
\end{align}
where we have used the assumption of $b$. It follows from  (\ref{s2.5})-(\ref{s2.7}) that there exists a unique $0<t_{u_i}<T_{u_i}$ such that
\begin{equation*}
g_i(t_{u_i})=0~\text{and}~h'_i(t_{u_i})<0,
\end{equation*}
which yields
\begin{equation*}
t_{u_i}^2\|u_i\|_i^2+t_{u_i}^4b\left(\displaystyle\int_{B_i}|\nabla u_i|^2{\rm d}x\right)^2-t_{u_i}^p\displaystyle\int_{B_i}|u_i|^p{\rm d}x=0
\end{equation*}
and
\begin{equation*}
(4-p)t_{u_i}^p\displaystyle\int_{B_i}|u_i|^p{\rm d}x<2t_{u_i}^2\|u_i\|_i^2,
\end{equation*}
for each $i=1,\cdots,k+1$. Therefore, we have a unique tuple $(t_1,\cdots,t_{k+1})\in(\mathbb{R}_{>0})^{k+1}$  such that $(t_1u_1,\cdots,t_{k+1}u_{k+1})\in N_k^-$.

\textbf{Step}\textbf{2:} We shall show that $Z$ is open in $[0,1]$. Suppose that $\mu_0\in Z$ and $(t^0_1,\cdots,t^0_{k+1})\in(\mathbb{R}_{>0})^{k+1}$  is the unique solution of (\ref{s2.2})-(\ref{s2.3}) with $\mu=\mu_0$. To see whether the Implicit Function Theorem can be applied at $\mu_0$, we calculate the matrix
\begin{equation*}
M=(M_{ij})=(\partial_{t_j}G_i)_{i,j=1,\cdots,k+1},
\end{equation*}
where $G_i$ denotes the left-hand side of (\ref{s2.2}). Then each component of the matrix is represented by
\begin{equation*}
M_{ii}=(4-p)(t^0_i)^{p-1}\displaystyle\int_{B_i}|u_i|^p{\rm d}x-2t^0_i\|u_i\|_i^2-2\mu_0bt^0_i\displaystyle\int_{B_i}|\nabla u_i|^2{\rm d}x\displaystyle\sum_{j\neq i}^{k+1}(t^0_j)^2\displaystyle\int_{B_j}|\nabla u_j|^2{\rm d}x
\end{equation*}
for $i=1,\cdots,k+1$, and
\begin{equation*}
M_{ij}=2\mu_0b(t^0_i)^2t^0_j\displaystyle\int_{B_i}|\nabla u_i|^2{\rm d}x\displaystyle\int_{B_j}|\nabla u_j|^2{\rm d}x,
\end{equation*}
for $i\neq j, i,j=1,\cdots,k+1$, where we have used (\ref{s2.2}). Therefore,
\begin{equation}\label{s2.8}
det M=\frac{(-1)^{k+1}}{t^0_1\cdots t^0_{k+1}}det\widetilde{M},
\end{equation}
where the matrix $\widetilde{M}=(\widetilde{M}_{i,j})$ is given by
\begin{equation*}
\widetilde{M}_{i,i}=-(4-p)(t^0_i)^p\displaystyle\int_{B_i}|u_i|^p{\rm d}x+2(t^0_i)^2\|u_i\|_i^2+2\mu_0b(t^0_i)^2\displaystyle\int_{B_i}|\nabla u_i|^2{\rm d}x\displaystyle\sum_{j\neq i}^{k+1}(t^0_j)^2\displaystyle\int_{B_j}|\nabla u_j|^2{\rm d}x,
\end{equation*}
for $i=1,\cdots,k+1$, and
\begin{equation*}
\widetilde{M}_{ij}=-2\mu_0b(t^0_i)^2(t^0_j)^2\displaystyle\int_{B_i}|\nabla u_i|^2{\rm d}x\displaystyle\int_{B_j}|\nabla u_j|^2{\rm d}x,
\end{equation*}
for $i\neq j, i,j=1,\cdots,k+1$. Thus,
\begin{equation*}
\displaystyle\sum_{j=1}^{k+1}\widetilde{M}_{ij}=-(4-p)(t^0_i)^p\displaystyle\int_{B_i}|u_i|^p{\rm d}x+2(t^0_i)^2\|u_i\|_i^2>0,
\end{equation*}
for $i=1,\cdots,k+1$, where have used (\ref{s2.3}). Hence the matrix $\widetilde{M}=(\widetilde{M}_{ij})$ is diagonally dominant, and so it is nonsingular, which together with (\ref{s2.8}) show that
\begin{equation*}
det M\neq0.
\end{equation*}
Then we can apply the Implicit Function Theorem to obtain a neighborhood $U_0$ of $\mu_0$ and $A_0\subset(\mathbb{R}_{>0})^{k+1}$ is a neighborhood of $(t^0_1,\cdots,t^0_{k+1})$ such the system (\ref{s2.2})-(\ref{s2.3}) is uniquely solvable in $U_0\times A_0$.

Suppose that there is $\mu_1\in U_0$ such that the second solution $(\overline{t}^0_1,\cdots,\overline{t}^0_{k+1})$ of (\ref{s2.2})-(\ref{s2.3}) exists in $(\mathbb{R}_{>0})^{k+1}\backslash A_0$, we deduce from the Implicit Function Theorem again that there exists a solution curve $(\mu,(\overline{t}^0_1(\mu),\cdots,\overline{t}^0_{k+1}(\mu)))$ in $(\mu_1-\epsilon,\mu_1+\epsilon)\times(\mathbb{R}_{>0})^{k+1}$ which satisfies (\ref{s2.2})-(\ref{s2.3}) and goes through $(\mu_1,(\overline{t}^0_1,\cdots,\overline{t}^0_{k+1}))$. Without loss of generality, we assume $\mu_0<\mu_1$ and extend this curve as long as possible. Due to it cannot be defined at $\mu_0$ and enter into $U_0\times A_0$, there is a point $\mu_2\in[\mu_0,\mu_1)$ such that $(t_1(\mu),\cdots,t_{k+1}(\mu))$ exists in $(\mu_2,\mu_1]$ and blows up as $\mu\rightarrow\mu_2^+$. However, this is impossible. Actually, for at least one $i$, the left-hand side of (\ref{s2.2}) is sufficiently large. Consequently, $U_0\subset Z$. The case $\mu_0>\mu_1$ is similar.

\textbf{Step}\textbf{3:} We shall show that $Z$ is closed in $[0,1]$. Suppose that there is a sequence $\{\mu_n\}\subset Z$ such that $\mu_n\rightarrow\mu_0\in[0,1]$
and $(t_1^n,\cdots,t_{k+1}^n)\in (\mathbb{R}_{>0})^{k+1}$ be the unique solution of  (\ref{s2.2})-(\ref{s2.3}) for $\mu_n$. Similarly as the preceding argument, we know that $(t_1^n,\cdots,t_{k+1}^n)$ is bounded. Then there exists a subsequence of $(t_1^n,\cdots,t_{k+1}^n)$ we still denote by it converges a solution $(t^0_1,\cdots,t^0_{k+1})\in (\mathbb{R}_+)^{k+1}$ of (\ref{s2.2}) for $\mu_0$ and
\begin{equation}\label{s2.9}
(4-p)(t^0_i)^p\displaystyle\int_{B_i}|u_i|^p{\rm d}x\leq2(t^0_i)^2\|u_i\|_i^2,
\end{equation}
i.e.,
\begin{equation}\label{s2.10}
t^0_i\leq\left(\frac{2\|u_i\|_i^2}{(4-p)\displaystyle\int_{B_i}|u_i|^p{\rm d}x}\right)^\frac{1}{p-2},
\end{equation}
for $i=1,\cdots,k+1$. Moreover, it follows from (\ref{s2.2}) that
\begin{equation}\label{s2.11}
(t^0_i)^2\|u_i\|_i^2\leq S_p^{-\frac{p}{2}}(t^0_i)^p\|u_i\|_i^p,
\end{equation}
for $i=1,\cdots,k+1$. Then we obtain that $t^0_i>0$ for $i=1,\cdots,k+1$. Consequently, $(t^0_1,\cdots,t^0_{k+1})\in (\mathbb{R}_{>0})^{k+1}$.

We claim that
\begin{equation}\label{s2.12}
(4-p)(t^0_i)^p\displaystyle\int_{B_i}|u_i|^p{\rm d}x<2(t^0_i)^2\|u_i\|_i^2,
\end{equation}
for $i=1,\cdots,k+1$. Suppose otherwise, we deduce from (\ref{s2.9}) that there exists at least one integer $i_0>0$ such that
\begin{equation}\label{s2.13}
(4-p)(t^0_{i_0})^p\displaystyle\int_{B_{i_0}}|u_{i_0}|^p{\rm d}x=2(t^0_{i_0})^2\|u_{i_0}\|_{i_0}^2.
\end{equation}
Moreover, using the Sobolev Inequality and the assumption of $u_i$, we deduce from (\ref{s2.9}) and (\ref{s2.11}) that
\begin{equation}\label{s2.15}
 (S_p)^\frac{p}{2(p-2)}\|u_i\|_i^{-1}\leq t^0_{i}\leq\left(\frac{2}{4-p}\right)^\frac{1}{p-2}(2S_p)^\frac{p}{2(p-2)}\|u_i\|_i^{-1},
\end{equation}
for $i=1,\cdots,k+1$.

On one hand, since $(t^0_1,\cdots,t^0_{k+1})\in (\mathbb{R}_{>0})^{k+1}$ is a solution of (\ref{s2.2}), we have
\begin{align}
&b\left(\displaystyle\int_{B_{i_0}}|\nabla u_{i_0}|^2{\rm d}x\right)^2-(t^0_{i_0})^{p-4}\displaystyle\int_{B_{i_0}}|u_{i_0}|^p{\rm d}x\notag
\\&+(t^0_{i_0})^{-2}\left(\|u_{i_0}\|_{i_0}^2+\mu_0b\displaystyle\int_{B_{i_0}}|\nabla u_{i_0}|^2{\rm d}x\displaystyle\sum_{j\neq i_0}^{k+1}(t^0_j)^2\displaystyle\int_{B_j}|\nabla u_j|^2{\rm d}x\right)=0.\label{s2.16}
\end{align}
On the other hand, with the help of  (\ref{s2.13}) and (\ref{s2.15}), it holds
\begin{align}
&~~~b\left(\displaystyle\int_{B_{i_0}}|\nabla u_{i_0}|^2{\rm d}x\right)^2+(t^0_{i_0})^{-2}\left(\|u_{i_0}\|_{i_0}^2+\mu_0b\displaystyle\int_{B_{i_0}}|\nabla u_{i_0}|^2{\rm d}x\displaystyle\sum_{j\neq i_0}^{k+1}(t^0_j)^2\displaystyle\int_{B_j}|\nabla u_j|^2{\rm d}x\right)\notag
\\&~~~-(t^0_{i_0})^{p-4}\displaystyle\int_{B_{i_0}}|u_{i_0}|^p{\rm d}x\notag
\\&\leq b\|u_{i_0}\|_{i_0}^4-\frac{p-2}{4-p}(t_{i_0}^0)^{-2}\|u_{i_0}\|_{i_0}^2+(t_{i_0}^0)^{-2}b\displaystyle\int_{B_{i_0}}|\nabla u_{i_0}|^2{\rm d}x\displaystyle\sum_{j\neq i_0}^{k+1}(t^0_j)^2\displaystyle\int_{B_j}|\nabla u_j|^2{\rm d}x\notag
\\&\leq b\|u_{i_0}\|_{i_0}^4-\frac{p-2}{4-p}\left(\frac{4-p}{2}\right)^\frac{2}{p-2}(2S_p)^{\frac{-p}{p-2}}\|u_{i_0}\|_{i_0}^4
+b(t_{i_0}^0)^{-2}\displaystyle\sum_{j\neq i_0}^{k+1}(t^0_j)^2\displaystyle\int_{B_j}|\nabla u_j|^2{\rm d}x\notag
\\&\leq\left(b\left(1+k2^\frac{p}{p-2}\left(\frac{2}{4-p}\right)^\frac{2}{p-2}\right)-\frac{p-2}{4-p}\left(\frac{4-p}{2}\right)^\frac{2}{p-2}(2S_p)^{\frac{-p}{p-2}}\right)\|u_{i_0}\|_{i_0}^4<0,\label{s2.17}
\end{align}
which contradicts (\ref{s2.16}), where we have used the assumption of $b$. Consequently, $(t^0_1,\cdots,t^0_{k+1})\in (\mathbb{R}_{>0})^{k+1}$ is a solution of (\ref{s2.2})-(\ref{s2.3}) for $\mu=\mu_0$. Furthermore, $(t^0_1,\cdots,t^0_{k+1})$ is a unique solution in $(\mathbb{R}_{>0})^{k+1}$ follows from the Implicit Function Theorem.

Consequently, $Z=[0,1]$ and  we complete the proof.
\end{proof}
\begin{lemma}\label{l2.2}
$N_k^-$ is a differentiable manifold in $\mathcal{\textbf{H}}_k$.  Moreover, the critical points of the restriction $E_b|_{N_k^-}$ of $E_b$ to $N_k^-$ are also critical points of $E_b$ with no zero component.
\end{lemma}
\begin{proof}
Let $\mathcal{F}=(F_1,\cdots,F_{k+1}):\mathcal{\textbf{H}}_k\rightarrow \mathbb{R}^{k+1}$ given by
\begin{align}
F_i(u_1,\cdots,u_{k+1})=&\|u_i\|_i^2+b\left(\displaystyle\int_{B_i}|\nabla u_i|^2{\rm d}x\right)^2\notag
\\&+b\displaystyle\int_{B_i}|\nabla u_i|^2{\rm d}x\displaystyle\sum_{j\neq i}^{k+1}\displaystyle\int_{B_j}|\nabla u_j|^2{\rm d}x-\displaystyle\int_{B_i}|u_i|^p{\rm d}x,\label{s2.18}
\end{align}
for $i=1,\cdots,k+1$. Then
\begin{align}
N_k^-=&\{(u_1,\cdots,u_{k+1})\in \mathcal{\textbf{H}}_k,u_i\neq0| F_i(u_1,\cdots,u_{k+1})=0,\notag
\\&(4-p)\displaystyle\int_{B_i}|u_i|^p{\rm d}x<2\|u_i\|_i^2,~\text{for}~i=1,\cdots,k+1\}.\label{s2.19}
\end{align}

To show that $N_k^-$ is differentiable in $\mathcal{\textbf{H}}_k$, we only need to prove that
\begin{equation*}
N:=(N_{i,j})=(\langle\partial_{u_i}F_j(u_1,\cdots,u_{k+1}),u_i\rangle)_{i,j=1,\cdots,k+1}
\end{equation*}
at each point $(u_1,\cdots,u_{k+1})\in N_k^-$ is nonsingular due to it means that $0$ is a regular value of $\mathcal{F}$. Standard calculation yields that
\begin{align}
N_{ii}&=2\|u_i\|_i^2+4b\left(\displaystyle\int_{B_i}|\nabla u_i|^2{\rm d}x\right)^2+2b\displaystyle\int_{B_i}|\nabla u_i|^2{\rm d}x\displaystyle\sum_{j\neq i}^{k+1}\displaystyle\int_{B_j}|\nabla u_j|^2{\rm d}x-p\displaystyle\int_{B_i}|u_i|^p{\rm d}x\notag
\\&=-2\|u_i\|_i^2-2b\displaystyle\int_{B_i}|\nabla u_i|^2{\rm d}x\displaystyle\sum_{j\neq i}^{k+1}\displaystyle\int_{B_j}|\nabla u_j|^2{\rm d}x+(4-p)\displaystyle\int_{B_i}|u_i|^p{\rm d}x,\notag
\end{align}
for $i=1,\cdots,k+1$, and
\begin{equation*}
N_{ij}=2b\displaystyle\int_{B_i}|\nabla u_i|^2{\rm d}x\displaystyle\int_{B_j}|\nabla u_j|^2{\rm d}x,
\end{equation*}
for $i\neq j$ and $i,j=1,\cdots,k+1$. Then
\begin{equation*}
\displaystyle \sum_{j=1}^{k+1}N_{ij}=-2\|u_i\|_i^2+(4-p)\displaystyle\int_{B_i}|u_i|^p{\rm d}x<0,
\end{equation*}
where we have used (\ref{s2.18}) and (\ref{s2.19}). It means that $N$ is diagonally dominant at each point $(u_1,\cdots,u_{k+1})\in N_k^-$, and so it is invertible.

Let $(u_1,\cdots,u_{k+1})$ be a critical point of $E_b|_{N_k^-}$, it follows that there exist Lagrange multipliers $\lambda_1,\cdots,\lambda_{k+1}$ such that
\begin{equation*}
\lambda_1 F_1'(u_1,\cdots,u_{k+1})+\cdots+\lambda_{k+1} F_{k+1}'(u_1,\cdots,u_{k+1})=E_b'(u_1,\cdots,u_{k+1}).
\end{equation*}
Taking $(u_1,0,\cdots,0),(u_2,0,0,\cdots,0),\cdots,(0,\cdots,0,u_{k+1})$ into the above identity and noticing (\ref{s2.19}), we obtain
\begin{equation*}
N(\lambda_1,\cdots,\lambda_{k+1})^T=(0,\cdots,0)^T.
\end{equation*}
Thus $\lambda_i=0$, for $i=1,\cdots,k+1$, and so $(u_1,\cdots,u_{k+1})$ is a critical point of $E_b$.

For any $(u_1,\cdots,u_{k+1})\in N_k^-$, since $\|u_i\|_i^2\leq \displaystyle\int_{B_i}|u_i|^p{\rm d}x\leq C\|u_i\|_i^p$ for some $C>0$, thus each $u_i$ is bounded away from zero. Then we obtain critical points of $E_b$ in $N_k^-$ have no zero component.
The proof is completed.
\end{proof}
\begin{lemma}\label{l2.3}
For each $(u_1,\cdots,u_{k+1})\in \mathcal{\textbf{H}}_k$, with
\begin{equation}\label{s2.20}
(4-p)\displaystyle\int_{B_i}|u_i|^p{\rm d}x<2\|u_i\|_i^2
\end{equation}
and
\begin{equation}\label{s2.21}
\|u_i\|_i^2+b\left(\displaystyle\int_{B_i}|\nabla u_i|^2{\rm d}x\right)^2+b\displaystyle\int_{B_i}|\nabla u_i|^2{\rm d}x\displaystyle\sum_{j\neq i}^{k+1}\displaystyle\int_{B_j}|\nabla u_j|^2{\rm d}x\leq \displaystyle\int_{B_i}|u_i|^p{\rm d}x,
\end{equation}
for $i=1,\cdots,k+1$,
there exists a unique $(k+1)$-tuple $(\widehat{t}_1,\cdots,\widehat{t}_{k+1})$ of positive numbers such that $(\widehat{t}_1u_1,\cdots,\widehat{t}_{k+1}u_{k+1})\in N_k^-$ with $\widehat{t}_i\leq1$ for $i=1,\cdots,k+1$. Moreover, we have
\begin{equation}\label{s2.22}
E_b(\widehat{t}_1u_1,\cdots,\widehat{t}_{k+1}u_{k+1})=\displaystyle\max_{0\leq t_i\leq1,\\i=1,\cdots,k+1}E_b(t_1u_1,\cdots,t_{k+1}u_{k+1}).
\end{equation}
\end{lemma}
\begin{proof}
Let
\begin{align}
\phi_i(t_1,\cdots,t_{k+1})&=t_i^2\|u_i\|_i^2+t_i^4b\left(\displaystyle\int_{B_i}|\nabla u_i|^2{\rm d}x\right)^2\notag
\\&~~~+t_i^2b\displaystyle\int_{B_i}|\nabla u_i|^2{\rm d}x\displaystyle\sum_{j\neq i}^{k+1}t_j^2\displaystyle\int_{B_j}|\nabla u_j|^2{\rm d}x-t_i^p \displaystyle\int_{B_i}|u_i|^p{\rm d}x,\notag
\end{align}
for $i=1,\cdots,k+1$. Then we deduce from (\ref{s2.21}) that there exists $r>0$ small enough such that
\begin{equation*}
\phi_i(t_1,\cdots,t_{k+1})>0,~\text{if}~t_i=r
\end{equation*}
and
\begin{equation*}
\phi_i(t_1,\cdots,t_{k+1})\leq0,~\text{if}~t_i=1,
\end{equation*}
for $i=1,\cdots,k+1$. Thus, by applying the Poincar\'{e}-Miranda Lemma (see \cite{26}), we get that there is a $(k+1)$-tuple $(\widehat{t}_1,\cdots,\widehat{t}_{k+1})$ of positive numbers with $0<\widehat{t}_i\leq1$ such that
\begin{equation*}
\phi_i(\widehat{t}_1,\cdots,\widehat{t}_{k+1})=0,
\end{equation*}
for $i=1,\cdots,k+1$, which together with (\ref{s2.20}) show that  $(\widehat{t}_1u_1,\cdots,\widehat{t}_{k+1}u_{k+1})\in N_k^-$.

Now, let us prove the uniqueness. We suppose that  $(u_1,\cdots,u_{k+1})\in N_k^-$ for simplicity. Then we have
\begin{equation}\label{s2.23}
(4-p)\displaystyle\int_{B_i}|u_i|^p{\rm d}x<2\|u_i\|_i^2
\end{equation}
and
\begin{equation}\label{s2.24}
\|u_i\|_i^2+b\left(\displaystyle\int_{B_i}|\nabla u_i|^2{\rm d}x\right)^2+b\displaystyle\int_{B_i}|\nabla u_i|^2{\rm d}x\displaystyle\sum_{j\neq i}^{k+1}\displaystyle\int_{B_j}|\nabla u_j|^2{\rm d}x= \displaystyle\int_{B_i}|u_i|^p{\rm d}x,
\end{equation}
for $i=1,\cdots,k+1$. If we have another $(k+1)$-tuple $(c_1,\cdots,c_{k+1})$ of positive numbers such that  $(c_1u_1,\cdots,c_{k+1}u_{k+1})\in N_k^-$, it is sufficient to prove that
\begin{equation}\label{s2.25}
(c_1,\cdots,c_{k+1})=(1,\cdots,1),
\end{equation}
in the following. Denote
\begin{equation*}
c_{i_0}=\max\{c_1,\cdots,c_{k+1}\}~\text{and}~c_{j_0}=\min\{c_1,\cdots,c_{k+1}\}.
\end{equation*}
Then we only need to show that $c_{i_0}\leq1$ and $c_{j_0}\geq1$.  Let
\begin{align}
\phi_{i_0}(t)&:=t^2\|u_{i_0}\|_{i_0}^2+bt^4\left(\displaystyle\int_{B_{i_0}}|\nabla u_{i_0}|^2{\rm d}x\right)^2\notag
\\&~~~~~~+bt^4\displaystyle\int_{B_{i_0}}|\nabla u_{i_0}|^2{\rm d}x\displaystyle\sum_{j\neq i_0}^{k+1}\displaystyle\int_{B_j}|\nabla u_j|^2{\rm d}x-t^p\displaystyle\int_{B_i}|u_{i_0}|^p{\rm d}x\notag
\\&=t^4\left[t^{-2}\|u_{i_0}\|_{i_0}^2+b\displaystyle\int_{B_{i_0}}|\nabla u_{i_0}|^2{\rm d}x\displaystyle\sum_{j=1}^{k+1}\displaystyle\int_{B_j}|\nabla u_j|^2{\rm d}x-t^{p-4}\displaystyle\int_{B_i}|u_{i_0}|^p{\rm d}x\right]\notag
\end{align}
and
\begin{equation*}
h_{i_0}(t):=t^{-2}\|u_{i_0}\|_{i_0}^2-t^{p-4}\displaystyle\int_{B_i}|u_{i_0}|^p{\rm d}x,
\end{equation*}
with $t>0$. Thus, (\ref{s2.24}) is equivalent to
\begin{equation}\label{s2.26}
h_{i_0}(1)+b\displaystyle\int_{B_{i_0}}|\nabla u_{i_0}|^2{\rm d}x\displaystyle\sum_{j=1}^{k+1}\displaystyle\int_{B_j}|\nabla u_j|^2{\rm d}x=0.
\end{equation}
Standard computation shows that
\begin{align}
h_{i_0}'(t)&=(-2)t^{-3}\|u_{i_0}\|_{i_0}^2+(4-p)t^{p-5}\displaystyle\int_{B_{i_0}}|u_{i_0}|^p{\rm d}x\notag
\\&=t^{-3}\left[(4-p)t^{p-2}\displaystyle\int_{B_i}|u_{i_0}|^p{\rm d}x-2\|u_{i_0}\|_{i_0}^2\right].\notag
\end{align}
Hence $h_{i_0}(t)$ is decreasing in $(0,T_{u_{i_0}})$ and increasing in $(T_{u_{i_0}},+\infty)$, where
\begin{equation*}
T_{u_{i_0}}:=\left(\frac{2\|u_{i_0}\|_{i_0}^2}{(4-p)\displaystyle\int_{B_i}|u_{i_0}|^p{\rm d}x}\right)^\frac{1}{p-2}>1.
\end{equation*}
 Noticing  (\ref{s2.26}) and $T_{u_{i_0}}>1$, it is easy to obtain that
\begin{equation*}
h_{i_0}(T_{u_{i_0}})+b\displaystyle\int_{B_{i_0}}|\nabla u_{i_0}|^2{\rm d}x\displaystyle\sum_{j=1}^{k+1}\displaystyle\int_{B_j}|\nabla u_j|^2{\rm d}x<0.
\end{equation*}
Consequently, there exists $0<1<T_{u_{i_0}}<s_{i_0}<+\infty$, such that $\phi_{i_0}(s_{i_0})=0$ and
\begin{equation}\label{s2.27}
\phi_{i_0}(t)\left\{\begin{array}{ll}
\geq0,~~0<t\leq 1,\\
\leq0,~~1<t<s_{i_0},\\
\geq0,~~s_{i_0}\leq t.
\end{array}\right.
\end{equation}
Moreover,
\begin{equation}\label{s2.28}
h_{i_0}(t)\left\{\begin{array}{ll}
\leq0,~~0<t\leq T_{u_{i_0}},\\
\geq0,~~T_{u_{i_0}}\leq t.
\end{array}\right.
\end{equation}

Since  $(c_1u_1,\cdots,c_{k+1}u_{k+1})\in N_k^-$, it follows that
\begin{align}
&~~~~c^2_{i_0}\|u_{i_0}\|_{i_0}^2+bc^4_{i_0}\left(\displaystyle\int_{B_{i_0}}|\nabla u_{i_0}|^2{\rm d}x\right)^2+bc^2_{i_0}\displaystyle\int_{B_{i_0}}|\nabla u_{i_0}|^2{\rm d}x\displaystyle\sum_{j\neq i_0}^{k+1}c^2_j\displaystyle\int_{B_j}|\nabla u_j|^2{\rm d}x\notag
\\&=c^p_{i_0}\displaystyle\int_{B_{i_0}}|u_{i_0}|^p{\rm d}x,\notag
\end{align}
and
\begin{equation*}
(4-p)c^p_{i_0}\displaystyle\int_{B_{i_0}}|u_{i_0}|^p{\rm d}x<2c^2_{i_0}\|u_{i_0}\|_{i_0}^2.
\end{equation*}
Therefore,
\begin{equation}\label{s2.29}
\phi_{i_0}(c_{i_0})\geq0~\text{and}~h'_{i_0}(c_{i_0})<0.
\end{equation}
Combining (\ref{s2.27})-(\ref{s2.29}), we can easily get $c_{i_0}\leq1$.

Similarly, since
\begin{align}
&~~~~c^2_{j_0}\|u_{j_0}\|_{j_0}^2+bc^4_{j_0}\left(\displaystyle\int_{B_{j_0}}|\nabla u_{j_0}|^2{\rm d}x\right)^2+bc^2_{j_0}\displaystyle\int_{B_{j_0}}|\nabla u_{j_0}|^2{\rm d}x\displaystyle\sum_{j\neq j_0}^{k+1}c^2_j\displaystyle\int_{B_j}|\nabla u_j|^2{\rm d}x\notag
\\&=c^p_{j_0}\displaystyle\int_{B_{j_0}}|u_{j_0}|^p{\rm d}x,\notag
\end{align}
and
\begin{equation*}
(4-p)\displaystyle\int_{B_{j_0}}|u_{j_0}|^p{\rm d}x<2\|u_{j_0}\|_{j_0}^2,
\end{equation*}
we obtain
\begin{equation*}
\phi_{j_0}(c_{j_0})\leq0~\text{and}~h'_{i_0}(c_{j_0})<0,
\end{equation*}
and so $c_{j_0}\geq1$. Then we obtain the uniqueness.

Finally, we will prove (\ref{s2.22}). Let $(\widetilde{t}_1,\cdots,\widetilde{t}_{k+1})\in(\mathbb{R}_{>0})^{k+1}$ be a critical point of $E_b$ with $\widetilde{t}_i\leq1$ $i=1,\cdots,k+1$, we have
\begin{equation*}
(4-p)\widetilde{t}_i^p\displaystyle\int_{B_i}|u_i|^p{\rm d}x<2\widetilde{t}_i^2\|u_i\|_i^2
\end{equation*}
and
\begin{equation*}
\widetilde{t}_i\|u_i\|_i^2+\widetilde{t}_i^3b\left(\displaystyle\int_{B_i}|\nabla u_i|^2{\rm d}x\right)^2+\widetilde{t}_ib\displaystyle\int_{B_i}|\nabla u_i|^2{\rm d}x\displaystyle\sum_{j\neq i}^{k+1}\widetilde{t}_j^2\displaystyle\int_{B_j}|\nabla u_j|^2{\rm d}x=\widetilde{t}_i^{p-1}\displaystyle\int_{B_i}|u_i|^p{\rm d}x,
\end{equation*}
for $i=1,\cdots,k+1$, which means that $(\widetilde{t}_1u_1,\cdots,\widetilde{t}_{k+1}u_{k+1})\in N_k^-$. Thus, there is only a unique critical point of
$\varphi(t_1,\cdots,t_{k+1})=E_b(t_1u_1,\cdots,t_{k+1}u_{k+1})$ with $0<t_i\leq1$ for $i=1,\cdots,k+1$.

Choose $(c^0_1,\cdots,c^0_{k+1})\in\partial(\mathbb{R}_{>0})^{k+1}$, we may assume that $c^0_1=0$ for simplicity. Then, since
\begin{align}
&~~~~\varphi(t,c^0_2,\cdots,c^0_{k+1})\notag
\\&=\frac{t^2}{2}\|u_1\|_1^2+\frac{bt^4}{4}\left(\displaystyle\int_{B_i}|\nabla u_i|^2{\rm d}x\right)^2\notag
\\&~~~~+\frac{bt^2}{2}\displaystyle\int_{B_1}|\nabla u_1|^2{\rm d}x\displaystyle\sum_{j=2}^{k+1}(c^0_j)^2)\displaystyle\int_{B_j}|\nabla u_j|^2{\rm d}x-\frac{t^p}{p} \displaystyle\int_{B_1}|u_1|^p{\rm d}x\notag
\\&~~~~+\displaystyle\sum_{i=2}^{k+1}\|c^0_iu_i\|_i^2+\frac{b}{4}\displaystyle\sum_{i,j=2}^{k+1}(c^0_ic^0_j)^2\displaystyle\int_{B_i}|\nabla u_i|^2{\rm d}x\displaystyle\int_{B_j}|\nabla u_j|^2{\rm d}x-\frac{1}{p}\displaystyle\sum_{i=2}^{k+1}\displaystyle\int_{B_i}|c^0_iu_i|^p{\rm d}x,\notag
\end{align}
is an increasing function with respect to $t$ is small enough, we obtain (\ref{s2.22}) based on the analysis above. The proof is completed.
\end{proof}

Combining the above results, we shall prove the following.
\begin{lemma}\label{l2.4}
Let
\begin{equation*}
b^*:=\min\left\{b_*,\frac{(p-2)^2}{8p(4-p)\alpha_k}\right\},
\end{equation*}
where $b_*$ is defined in Lemma \ref{l2.1} and
\begin{equation*}
\alpha_k:=\displaystyle\min_{(u_1,\cdots,u_{k+1})\in N_k^-}E_b(u_1,\cdots,u_{k+1})\geq \frac{(k+1)(p-2)}{4p}(S_p)^\frac{p}{p-2}.
\end{equation*}
Then, for any fixed $\overrightarrow{\mathbf{r}}_k=(r_1,\cdots,r_k)\in\Lambda_k$ and $b\in(0,b^*)$,
there is a minimizer $(\omega_1,\cdots,\omega_{k+1})$ of its corresponding energy $E_b|_{N_k^-}$ such that $(-1)^{i+1}\omega_i$ is positive in $B_i$ for $i=1,\cdots,k+1$. Moreover, it satisfies (\ref{s2.1}).
\end{lemma}
\begin{proof}
For any $(u_1,\cdots,u_{k+1})\in N_k^-$, it follows from the definition of $N_k^-$ and the Sobolev Inequality that
\begin{equation*}
\|u_i\|_i^2\leq\displaystyle\int_{B_i}|u_i|^p{\rm d}x\leq(S_p)^{-\frac{p}{2}}\|u_i\|_i^p,
\end{equation*}
for $i=1,\cdots,k+1$.
Then we obtain
\begin{equation}\label{s2.30}
\|u_i\|_i\geq(S_p)^\frac{p}{2(p-2)},
\end{equation}
for $i=1,\cdots,k+1$, and so
\begin{align}
E_b(u_1,\cdots,u_{k+1})&\geq\frac{1}{4}\displaystyle\sum_{i=1}^{k+1}\|u_i\|_i^2-\left(\frac{1}{p}-\frac{1}{4}\right)\displaystyle\sum_{i=1}^{k+1}\displaystyle\int_{B_i}|u_i|^p{\rm d}x\notag
\\&\geq\frac{p-2}{4p}\displaystyle\sum_{i=1}^{k+1}\|u_i\|_i^2\geq\frac{(k+1)(p-2)}{4p}(S_p)^\frac{p}{p-2},\label{s2.31}
\end{align}
which implies that
\begin{equation*}
\alpha_k\geq \frac{(k+1)(p-2)}{4p}(S_p)^\frac{p}{p-2}.
\end{equation*}
Let $\{(u^n_1,\cdots,u^n_{k+1})\}\subset N_k^-$ be a minimizing sequence of $E_b|_{N_k^-}$, we deduce from (\ref{s2.31}) that it is bounded in  $\mathcal{\textbf{H}}_k$.
Thus, we may assume that $(u^n_1,\cdots,u^n_{k+1})$ converges weakly to some element $(u^0_1,\cdots,u^0_{k+1})$ in  $\mathcal{\textbf{H}}_k$.

We claim that $u^0_i\neq0$ for $i=1,\cdots,k+1$. Actually, noticing that
\begin{equation}\label{e2.31}
\|u^0_i\|_i\leq \displaystyle\liminf_{n\rightarrow\infty}\|u^n_i\|_i,
\end{equation}
and
\begin{equation}\label{e2.32}
\displaystyle\lim_{n\rightarrow\infty}\displaystyle\int_{B_i}|u^n_i|^p{\rm d}x=\displaystyle\int_{B_i}|u^0_i|^p{\rm d}x,
\end{equation}
 for $i=1,\cdots,k+1$,
ones obtain from the definition of $N_k^-$ that
\begin{equation*}
\|u^0_i\|_i^2\leq\displaystyle\int_{B_i}|u^0_i|^p{\rm d}x\leq(S_p)^{-\frac{p}{2}}\|u^0_i\|_i^p,
\end{equation*}
for $i=1,\cdots,k+1$, and so
\begin{equation*}
\|u^0_i\|_i\geq(S_p)^\frac{p}{2(p-2)},
\end{equation*}
for $i=1,\cdots,k+1$, where we have used the compact Sobolev embedding $H_r^1(\mathbb{R}^3)\hookrightarrow L^q(\mathbb{R}^3)$, $q\in(2,6)$. Then we obtain the claim.

As for
\begin{equation}\label{s2.32}
(u^n_1,\cdots,u^n_{k+1})\rightarrow (u^0_1,\cdots,u^0_{k+1})
\end{equation}
in $\mathcal{\textbf{H}}_k$, we shall prove it by way of contradiction. Suppose that (\ref{s2.32}) is not true, we have
\begin{equation}\label{s2.33}
\|u^0_{i_0}\|_{i_0}<\displaystyle\liminf_{n\rightarrow\infty}\|u^n_{i_0}\|_{i_0},
\end{equation}
for at least one $i_0\in\{1,\cdots,k+1\}$. Let
\begin{align}
\widetilde{\phi}_i(t_1,\cdots,t_{k+1})&=t_i^2\|u^0_i\|_i^2+t_i^4b\left(\displaystyle\int_{B_i}|\nabla u^0_i|^2{\rm d}x\right)^2\notag
\\&~~~~+t_i^2b\displaystyle\int_{B_i}|\nabla u^0_i|^2{\rm d}x\displaystyle\sum_{j\neq i}^{k+1}t_j^2\displaystyle\int_{B_j}|\nabla u^0_j|^2{\rm d}x-t_i^p \displaystyle\int_{B_i}|u^0_i|^p{\rm d}x,\notag
\end{align}
for $i=1,\cdots,k+1$. Then, similarly as the proof  of Lemma \ref{l2.3}, we deduce from the Poincar\'{e}-Miranda Lemma (see\cite{26}) that there exists $(t_1^0,\cdots,t_{k+1}^0)\in(\mathbb{R}_{>0})^{k+1}$ such that
\begin{equation}\label{s2.34}
\widetilde{\phi}_i(t^0_1,\cdots,t^0_{k+1})=0,
 \end{equation}
 with $t^0_i\leq1$ for $i=1,\cdots,k+1$. Moreover, it is easy to check that $(t_1^0,\cdots,t_{k+1}^0)\neq(1,\cdots,1)$ by (\ref{s2.33}).

We claim that
\begin{equation}\label{s2.35}
(4-p)(t^0_i)^p\displaystyle\int_{B_i}|u^0_i|^p{\rm d}x<(t^0_i)^22\|u^0_i\|_i^2,
 \end{equation}
for $i=1,\cdots,k+1$. Otherwise, we obtain from (\ref{s2.34}) that
\begin{equation*}
\|t^0_iu^0_i\|_i^2+b\left(\displaystyle\int_{B_i}|t^0_i\nabla u^0_i|^2{\rm d}x\right)^2+b\displaystyle\int_{B_i}|t^0_i\nabla u^0_i|^2{\rm d}x\displaystyle\sum_{j\neq i}^{k+1}\displaystyle\int_{B_j}|t^0_j\nabla u^0_j|^2{\rm d}x\geq\frac{2}{4-p}\|t^0_iu^0_i\|_i^2,
\end{equation*}
which together with (\ref{s2.31}) yields
\begin{equation*}
\frac{p-2}{4-p}\|t^0_iu^0_i\|_i^2\leq b\|t^0_iu^0_i\|_i^2\displaystyle\sum_{j\neq i}^{k+1}\displaystyle\int_{B_j}|t^0_j\nabla u^0_j|^2{\rm d}x\leq \frac{8bp\alpha_k}{p-2}\|t^0_iu^0_i\|_i^2,
\end{equation*}
and so
\begin{equation*}
b\geq\frac{(p-2)^2}{8p(4-p)\alpha_k}.
\end{equation*}
This contradicts the assumption of $b$. Then we obtain (\ref{s2.35}).

Combining (\ref{s2.34}) and (\ref{s2.35}), we see that
$(t^0_1u^0_1,\cdots,t^0_{k+1}u^0_{k+1})\in N_k^-$. Therefore, notice that (\ref{e2.31})-(\ref{e2.32}), $(t_1^0,\cdots,t_{k+1}^0)\neq(1,\cdots,1)$ and Lemma \ref{l2.3}, we have
\begin{align}
\alpha_k&=\displaystyle\lim_{n\rightarrow\infty}E_b(u^n_1,\cdots,u^n_{k+1})\notag
\\&\geq\displaystyle\lim_{n\rightarrow\infty}E_b(t^0_1u^n_1,\cdots,t^0_{k+1}u^n_{k+1})\notag
\\&>E_b(t^0_1u^0_1,\cdots,t^0_{k+1}u^0_{k+1})\geq\alpha_k.\notag
\end{align}
That is obviously a contradiction. Hence, $(u^0_1,\cdots,u^0_{k+1})$ is a minimizer of $E_b|_{N_k^-}$. Clearly, we see that
\begin{equation*}
(\omega_1,\cdots,\omega_{k+1}):=(|u^0_1|,-|u^0_2|,\cdots,(-1)^{k+2}|u^0_{k+1}|)
\end{equation*}
is also a minimizer of $E_b|_{N_k^-}$, and by Lemma \ref{l2.2}, actually a critical point of $E_b$. Furthermore, noticing that it satisfies (\ref{s2.1}), ones deduce that each $(-1)^{i+1}\omega_i$ is positive in $B_i$ due to the strong maximum principle (see\cite{27}).
\end{proof}

\section{\normalsize{Existence of sign-changing radial solutions}}
\quad\quad
In this section, we shall find a least energy radial solution of (\ref{s2.1}) among elements in $\Lambda_k$. By using it as a building block, a radial solution of
(\ref{s1.1}) that changes sign exactly $k$ times will be established. From then on, we will attach a superscript $\overrightarrow{\textbf{r}}_k$ on a notion to emphasize the dependence of it on  $\overrightarrow{\textbf{r}}_k$.

For any $k$-tuple $\overrightarrow{\textbf{r}}_k=(r_1,\cdots,r_k)\in\Lambda_k$, we see that there exists a solution $\mathcal{\omega}^{\overrightarrow{\textbf{r}}_k}=(\omega_1^{\overrightarrow{\textbf{r}}_k},\cdots,\omega_{k+1}^{\overrightarrow{\textbf{r}}_k})$ of (\ref{s2.1})  which consists of sign changing components from the section above. To compare the energy of them, we define the function $\varphi:\Lambda_k\rightarrow \mathbb{R}$ by
\begin{align}
\varphi(\overrightarrow{\textbf{r}}_k)&=\varphi(r_1,\cdots,r_k)=E_b^{\overrightarrow{\textbf{r}}_k}(\omega_1^{\overrightarrow{\textbf{r}}_k},\cdots,\omega_{k+1}^{\overrightarrow{\textbf{r}}_k})\notag
\\&=\displaystyle\inf_{(u_1^{\overrightarrow{\textbf{r}}_k},\cdots,u_{k+1}^{\overrightarrow{\textbf{r}}_k})\in N_k^{\overrightarrow{\textbf{r}}_k,-}}E_b^{\overrightarrow{\textbf{r}}_k}(u_1^{\overrightarrow{\textbf{r}}_k},\cdots,u_{k+1}^{\overrightarrow{\textbf{r}}_k}),\label{s3.1}
\end{align}
and then give the following results.
\begin{lemma}\label{l3.1}
Let $\overrightarrow{\textbf{r}}_k=(r_1,\cdots,r_k)\in\Lambda_k$, there holds
\begin{description}
  \item[(i)] If $r_{i_0}-r_{{i_0}-1}\rightarrow0$ for some $i_0\in\{1,\cdots,k\}$, then $\varphi(\overrightarrow{\textbf{r}}_k)\rightarrow+\infty$.
  \item[(ii)] If $r_k\rightarrow\infty$, then $\varphi(\overrightarrow{\textbf{r}}_k)\rightarrow+\infty$.
  \item[(iii)] $\varphi$ is continuous in $\Lambda_k$.
\end{description}
In particular, there is a minimizer $\overrightarrow{\overline{{\textbf{r}}}}_k=(\overline{r}_1,\cdots,\overline{r}_k)\in\Lambda_k$ of $\varphi$.
\end{lemma}
\begin{proof}
(i). Since $(\omega_1^{\overrightarrow{\textbf{r}}_k},\cdots,\omega_{k+1}^{\overrightarrow{\textbf{r}}_k})\in N_k^{\overrightarrow{\textbf{r}}_k,-}$, we obtain by applying H\"{o}lder Inequality and Sobolev Inequality that
\begin{align}
\|\omega_{i_0}^{\overrightarrow{\textbf{r}}_k}\|_{i_0}^2\leq\displaystyle\int_{B^{\overrightarrow{\textbf{r}}_k}_{i_0}}|\omega_{i_0}^{\overrightarrow{\textbf{r}}_k}|^p{\rm d}x\leq\left(\displaystyle\int_{B^{\overrightarrow{\textbf{r}}_k}_{i_0}}|\omega_{i_0}^{\overrightarrow{\textbf{r}}_k}|^6{\rm d}x\right)^\frac{p}{6}|B^{\overrightarrow{\textbf{r}}_k}_{i_0}|^{1-\frac{p}{6}}
\leq C\|\omega_{i_0}^{\overrightarrow{\textbf{r}}_k}\|_{i_0}^p|B^{\overrightarrow{\textbf{r}}_k}_{i_0}|^{1-\frac{p}{6}},\notag
\end{align}
for some constant $C>0$, which means that $|B^{\overrightarrow{\textbf{r}}_k}_{i_0}|^{\frac{p}{6}-1}\leq C\|\omega_{i_0}^{\overrightarrow{\textbf{r}}_k}\|_{i_0}^{p-2}$.
Thus
\begin{equation*}
\|\omega_{i_0}^{\overrightarrow{\textbf{r}}_k}\|_{i_0}\rightarrow\infty,
\end{equation*}
as  $r_{i_0}-r_{{i_0}-1}\rightarrow0$. This together with (\ref{s2.31}) show that the first item follows.

(ii). For any $u\in H_r^1(\mathbb{R}^3)$, it follows from the Strauss Inequality (See [3]) that
\begin{equation*}
|u(x)|\leq C\frac{\|u\|_{\mathcal{H}}}{|x|},~a.e.~\text{in}~\mathbb{R}^3
\end{equation*}
with some positive constant $C$. By the above inequality and the fact $(\omega_1^{\overrightarrow{\textbf{r}}_k},\cdots,\omega_{k+1}^{\overrightarrow{\textbf{r}}_k})\in N_k^-$, ones have
\begin{align}
\|\omega_{k+1}^{\overrightarrow{\textbf{r}}_k}\|_{k+1}^2\leq\displaystyle\int_{B^{\overrightarrow{\textbf{r}}_k}_{k+1}}|\omega_{k+1}^{\overrightarrow{\textbf{r}}_k}|^p{\rm d}x\leq
\displaystyle\int_{B^{\overrightarrow{\textbf{r}}_k}_{k+1}}\frac{\|\omega_{k+1}^{\overrightarrow{\textbf{r}}_k}\|_{k+1}^p}{|x|^p}{\rm d}x=Cr_k^{3-p}\|\omega_{k+1}^{\overrightarrow{\textbf{r}}_k}\|_{k+1}^p,\notag
\end{align}
and so $r_k^{p-3}\leq C\|\omega_{k+1}^{\overrightarrow{\textbf{r}}_k}\|_{k+1}^{p-2}$. Hence the second item holds.

(iii). Suppose that a sequence
\begin{equation*}
\{\overrightarrow{\textbf{r}}_k^n\}_n=\{(r_1^n,\cdots,r_k^n)\}_n\subset\Lambda_k
\end{equation*}
converging to $\overrightarrow{\textbf{r}}_k=(r_1,\cdots,r_k)\in \Lambda_k$. We shall prove the following two aspects:
\begin{equation}\label{s3.2}
\varphi(\overrightarrow{\textbf{r}}_k)\geq\displaystyle\limsup_{n\rightarrow\infty}\varphi(\overrightarrow{\textbf{r}}^n_k)~\text{and}~\varphi(\overrightarrow{\textbf{r}}_k)\leq\displaystyle\liminf_{n\rightarrow\infty}\varphi(\overrightarrow{\textbf{r}}^n_k).
\end{equation}

Firstly, to show the former case: $\varphi(\overrightarrow{\textbf{r}}_k)\geq\displaystyle\limsup_{n\rightarrow\infty}\varphi(\overrightarrow{\textbf{r}}^n_k)$, we define $v_i^{\overrightarrow{\textbf{r}}^n_k}:[r_{i-1}^n,r_i^n]\rightarrow \mathbb{R}$ by
\begin{equation*}
v_i^{\overrightarrow{\textbf{r}}^n_k}:=\omega_i^{\overrightarrow{\textbf{r}}_k}\left(\frac{r_i-r_{i-1}}{r_i^n-r_{i-1}^n}(t-r_{i-1}^n)+r_{i-1}\right),~\text{for}~i=1,\cdots,k,
\end{equation*}
and
\begin{equation*}
v_{k+1}^{\overrightarrow{\textbf{r}}^n_k}:=\omega_i^{\overrightarrow{\textbf{r}}_k}\left(\frac{r_k}{r_k^n}t\right).
\end{equation*}
Then, standard calculation shows that
\begin{align}
&\|v_i^{\overrightarrow{\textbf{r}}^n_k}\|_{B_i^{\overrightarrow{\textbf{r}}^n_k}}^2=\|\omega_i^{\overrightarrow{\textbf{r}}_k}\|_{B_i^{\overrightarrow{\textbf{r}}_k}}^2+o_n(1),\label{s3.3}
\\&\displaystyle\int_{B^{\overrightarrow{\textbf{r}}^n_k}_i}|\nabla v_i^{\overrightarrow{\textbf{r}}^n_k}|^2{\rm d}x\displaystyle\int_{B^{\overrightarrow{\textbf{r}}^n_k}_j}|\nabla v_j^{\overrightarrow{\textbf{r}}^n_k}|^2{\rm d}x=\displaystyle\int_{B^{\overrightarrow{\textbf{r}}_k}_i}|\nabla \omega_i^{\overrightarrow{\textbf{r}}_k}|^2{\rm d}x\displaystyle\int_{B^{\overrightarrow{\textbf{r}}_k}_j}|\nabla \omega_j^{\overrightarrow{\textbf{r}}_k}|^2{\rm d}x+o_n(1),\label{s3.4}
\\&\displaystyle\int_{B^{\overrightarrow{\textbf{r}}^n_k}_i}|v_i^{\overrightarrow{\textbf{r}}^n_k}|^p{\rm d}x=\displaystyle\int_{B^{\overrightarrow{\textbf{r}}_k}_i}| \omega_i^{\overrightarrow{\textbf{r}}_k}|^p{\rm d}x+o_n(1),\label{s3.5}
\end{align}
for $i=1,\cdots,k+1$. Consequently, analogous to the argument of Step 2 in Lemma \ref{l2.1}, we obtain that there is a unique $(k+1)$-tuple of positive numbers $(t_1^n,\cdots,t_{k+1}^n)$ such that $(t_1^nv_1^{\overrightarrow{\textbf{r}}^n_k},\cdots,t_{k+1}^nv_{k+1}^{\overrightarrow{\textbf{r}}^n_k})\in  N_k^{\overrightarrow{\textbf{r}}^n_k,-}$. Furthermore, we have
\begin{equation}\label{s3.6}
\displaystyle\lim_{n\rightarrow\infty}t_i^n=1,
\end{equation}
for $i=1,\cdots,k+1$. By the definition of $(\omega_1^{\overrightarrow{\textbf{r}}^n_k},\cdots,\omega_{k+1}^{\overrightarrow{\textbf{r}}^n_k})$, it is easy to see
\begin{equation}\label{s3.7}
E_b^{\overrightarrow{\textbf{r}}^n_k}(t_1^nv_1^{\overrightarrow{\textbf{r}}^n_k},\cdots,t_{k+1}^nv_{k+1}^{\overrightarrow{\textbf{r}}^n_k})\geq E_b^{\overrightarrow{\textbf{r}}^n_k}(\omega_1^{\overrightarrow{\textbf{r}}^n_k},\cdots,\omega_{k+1}^{\overrightarrow{\textbf{r}}^n_k})=\varphi(\overrightarrow{\textbf{r}}_k^n).
\end{equation}
Combining (\ref{s3.4})-(\ref{s3.7}), it follows that
\begin{align}
\varphi(\overrightarrow{\textbf{r}}_k)&=E_b^{\overrightarrow{\textbf{r}}_k}(\omega_1^{\overrightarrow{\textbf{r}}_k},\cdots,\omega_{k+1}^{\overrightarrow{\textbf{r}}_k})=\displaystyle\limsup_{n\rightarrow\infty}E_b^{\overrightarrow{\textbf{r}}_k^n}(t_1^nv_1^{\overrightarrow{\textbf{r}}^n_k},\cdots,t_{k+1}^nv_{k+1}^{\overrightarrow{\textbf{r}}^n_k})\notag
\\&\geq\displaystyle\limsup_{n\rightarrow\infty}E_b^{\overrightarrow{\textbf{r}}_k^n}(\omega_1^{\overrightarrow{\textbf{r}}^n_k},\cdots,\omega_{k+1}^{\overrightarrow{\textbf{r}}^n_k})=\displaystyle\limsup_{n\rightarrow\infty}\varphi(\overrightarrow{\textbf{r}}_k^n).\label{s3.8}
\end{align}
We complete the proof of the first part in (\ref{s3.2}).

Secondly, to prove the latter case:$\varphi(\overrightarrow{\textbf{r}}_k)\leq\displaystyle\liminf_{n\rightarrow\infty}\varphi(\overrightarrow{\textbf{r}}^n_k)$, we define $v_i^{\overrightarrow{\textbf{r}}^n_k}:[r_{i-1},r_i]\rightarrow\mathbb{R}$ by
\begin{equation*}
v_i^{\overrightarrow{\textbf{r}}^n_k}:=\omega_i^{\overrightarrow{\textbf{r}}_k}\left(\frac{r_i^n-r_{i-1}^n}{r_i-r_{i-1}}(t-r_{i-1})+r^n_{i-1}\right),~\text{for}~i=1,\cdots,k,
\end{equation*}
and
\begin{equation*}
v_{k+1}^{\overrightarrow{\textbf{r}}^n_k}:=\omega_i^{\overrightarrow{\textbf{r}}_k^n}\left(\frac{r_k^n}{r_k}t\right).
\end{equation*}
where $r_0^n=0$ and $r_{k+1}^n=+\infty$. Then, standard computation shows that
\begin{align}
&\|v_i^{\overrightarrow{\textbf{r}}^n_k}\|_{B_i^{\overrightarrow{\textbf{r}}_k}}^2=\|\omega_i^{\overrightarrow{\textbf{r}}^n_k}\|_{B_i^{\overrightarrow{\textbf{r}}^n_k}}^2+o_n(1),\label{s3.9}
\\&\displaystyle\int_{B^{\overrightarrow{\textbf{r}}_k}_i}|\nabla v_i^{\overrightarrow{\textbf{r}}^n_k}|^2{\rm d}x\displaystyle\int_{B^{\overrightarrow{\textbf{r}}_k}_j}|\nabla v_j^{\overrightarrow{\textbf{r}}^n_k}|^2{\rm d}x=\displaystyle\int_{B^{\overrightarrow{\textbf{r}}^n_k}_i}|\nabla \omega_i^{\overrightarrow{\textbf{r}}^n_k}|^2{\rm d}x\displaystyle\int_{B^{\overrightarrow{\textbf{r}}^n_k}_j}|\nabla \omega_j^{\overrightarrow{\textbf{r}}^n_k}|^2{\rm d}x+o_n(1),\label{s3.10}
\\&\displaystyle\int_{B^{\overrightarrow{\textbf{r}}_k}_i}|v_i^{\overrightarrow{\textbf{r}}^n_k}|^p{\rm d}x=\displaystyle\int_{B^{\overrightarrow{\textbf{r}}^n_k}_i}| \omega_i^{\overrightarrow{\textbf{r}}^n_k}|^p{\rm d}x+o_n(1),\label{s3.11}
\end{align}
for $i=1,\cdots,k+1$. Combining (\ref{s3.8}) and the definition of $(\omega_1^{\overrightarrow{\textbf{r}}^n_k},\cdots,\omega_{k+1}^{\overrightarrow{\textbf{r}}^n_k})$, we deduce from (\ref{s2.31}) that
\begin{equation*}
\displaystyle\limsup_{n\rightarrow\infty}\displaystyle\sum_{i=1}^n\|\omega_i^{\overrightarrow{\textbf{r}}^n_k}\|_{B_i^{\overrightarrow{\textbf{r}}^n_k}}^2\leq\frac{4p\alpha_k}{p-2},
\end{equation*}
and
\begin{equation}\label{s3.12}
\|\omega_i^{\overrightarrow{\textbf{r}}^n_k}\|_{B_i^{\overrightarrow{\textbf{r}}^n_k}}^2+b\displaystyle\int_{B^{\overrightarrow{\textbf{r}}^n_k}_i}|\nabla \omega_i^{\overrightarrow{\textbf{r}}^n_k}|^2{\rm d}x\displaystyle\sum_{j=1}^{k+1}\displaystyle\int_{B^{\overrightarrow{\textbf{r}}^n_k}_j}|\nabla \omega_j^{\overrightarrow{\textbf{r}}^n_k}|^2{\rm d}x=\displaystyle\int_{B^{\overrightarrow{\textbf{r}}^n_k}_i}| \omega_i^{\overrightarrow{\textbf{r}}^n_k}|^p{\rm d}x,
\end{equation}
and
\begin{equation*}
\displaystyle\lim_{n\rightarrow\infty}(4-p)\displaystyle\int_{B^{\overrightarrow{\textbf{r}}^n_k}_i}| \omega_i^{\overrightarrow{\textbf{r}}^n_k}|^p{\rm d}x\leq\displaystyle\lim_{n\rightarrow\infty}2\|\omega_i^{\overrightarrow{\textbf{r}}^n_k}\|_{B_i^{\overrightarrow{\textbf{r}}^n_k}}^2.
\end{equation*}
Furthermore, similarly as the proof of (\ref{s2.35}), we can show
\begin{equation}\label{s3.13}
\displaystyle\lim_{n\rightarrow\infty}(4-p)\displaystyle\int_{B^{\overrightarrow{\textbf{r}}^n_k}_i}| \omega_i^{\overrightarrow{\textbf{r}}^n_k}|^p{\rm d}x<\displaystyle\lim_{n\rightarrow\infty}2\|\omega_i^{\overrightarrow{\textbf{r}}^n_k}\|_{B_i^{\overrightarrow{\textbf{r}}^n_k}}^2.
\end{equation}
Since (\ref{s3.12}) and (\ref{s3.13}) and noticing (\ref{s3.9}) and (\ref{s3.11}), we deduce from the Step 2 in the proof of Lemma  \ref{l2.1} that there is a unique $(k+1)$-tuple of positive numbers $(t_1^n,\cdots,t_{k+1}^n)$ such that
\begin{align}
&~~~~(t_i^n)^2\|v_i^{\overrightarrow{\textbf{r}}^n_k}\|_{B_i^{\overrightarrow{\textbf{r}}_k}}^2+b(t_i^n)^2\displaystyle\int_{B^{\overrightarrow{\textbf{r}}_k}_i}|\nabla v_i^{\overrightarrow{\textbf{r}}^n_k}|^2{\rm d}x\displaystyle\sum_{j=1}^{k+1}(t_j^n)^2\displaystyle\int_{B^{\overrightarrow{\textbf{r}}_k}_j}|\nabla v_j^{\overrightarrow{\textbf{r}}^n_k}|^2{\rm d}x\notag
\\&=(t_i^n)^p\displaystyle\int_{B^{\overrightarrow{\textbf{r}}_k}_i}|v_i^{\overrightarrow{\textbf{r}}^n_k}|^p{\rm d}x\notag
\end{align}
and
\begin{equation*}
(4-p)(t_i^n)^p\displaystyle\int_{B^{\overrightarrow{\textbf{r}}_k}_i}|v_i^{\overrightarrow{\textbf{r}}^n_k}|^p{\rm d}x<2(t_i^n)^2\|v_i^{\overrightarrow{\textbf{r}}^n_k}\|_{B_i^{\overrightarrow{\textbf{r}}_k}}^2,
\end{equation*}
for $i=1,\cdots,k+1$ as $n\rightarrow\infty$.  Moreover, we have
\begin{equation*}
\displaystyle\lim_{n\rightarrow\infty}t_i^n=1,
\end{equation*}
for $i=1,\cdots,k+1$. Consequently, we get
\begin{align}
\varphi(\overrightarrow{\textbf{r}}_k)&=E_b^{\overrightarrow{\textbf{r}}_k}(\omega_1^{\overrightarrow{\textbf{r}}_k},\cdots,\omega_{k+1}^{\overrightarrow{\textbf{r}}_k})\leq\displaystyle\liminf_{n\rightarrow\infty}E_b^{\overrightarrow{\textbf{r}}_k}(v_1^{\overrightarrow{\textbf{r}}^n_k},\cdots,v_{k+1}^{\overrightarrow{\textbf{r}}^n_k})\notag
\\&=\displaystyle\liminf_{n\rightarrow\infty}E_b^{\overrightarrow{\textbf{r}}^n_k}(\omega_1^{\overrightarrow{\textbf{r}}^n_k},\cdots,\omega_{k+1}^{\overrightarrow{\textbf{r}}^n_k})=\displaystyle\liminf_{n\rightarrow\infty}\varphi(\overrightarrow{\textbf{r}}_k^n).\notag
\end{align}
Then we complete the proof of item (iii).
\end{proof}

We shall prove that the point $\overrightarrow{\overline{\textbf{r}}}_k=(\overline{r}_1,\cdots,\overline{r}_k)\in\Lambda_k$ found in the above Lemma is the very element in $\Lambda_k$ which gives the solution of (\ref{s1.1}) with desired sign-changing property. For simplicity, we will use $t$ to denote the radial variable of functions  in $\mathcal{\textbf{H}}_k$.
\begin{proof}[Proof of Theorem 1.1.]
To the contrary, we assume that $\displaystyle\sum_{i=1}^{k+1}\omega^{\overrightarrow{\overline{\textbf{r}}}_k}$ is not a solution of (\ref{s1.1}). That is, there exists at least one $i_0\in\{1,\cdots,k\}$ such that
\begin{equation*}
\omega_-:=\displaystyle\lim_{t\rightarrow\overline{r}_{i_0}^-}\frac{d\omega_{i_0}^{\overrightarrow{\overline{\textbf{r}}}_k}(t)}{dt}\neq\displaystyle\lim_{t\rightarrow\overline{r}_{i_0}^+}\frac{d\omega_{i_0+1}^{\overrightarrow{\overline{\textbf{r}}}_k}(t)}{dt}=:\omega_+.
\end{equation*}
For convenience, we will drop the superscript $\overrightarrow{\overline{\textbf{r}}}_k$ in $(\omega_1^{\overrightarrow{\overline{\textbf{r}}}_k},\cdots,\omega_{k+1}^{\overrightarrow{\overline{\textbf{r}}}_k})$ and use $'$ to denote differentiation with respect to the radial variable $t$.

Fix a small $\sigma>0$, and define
\begin{equation}\label{s3.14}
\widehat{z}(t)=\left\{\begin{array}{ll}
\omega_{i_0}(t),~t\in(\overline{r}_{i_0-1},\overline{r}_{i_0}-\sigma),\\
z(t),~~t\in(\overline{r}_{i_0}-\sigma,\overline{r}_{i_0}+\sigma),\\
\omega_{i_0+1}(t),~t\in(\overline{r}_{i_0}+\sigma,\overline{r}_{i_0+1}),
\end{array}\right.
\end{equation}
where $z(t):=\omega_{i_0}(\overline{r}_{i_0}-\sigma)+\frac{\omega_{i_0+1}(\overline{r}_{i_0}+\sigma)-\omega_{i_0}(\overline{r}_{i_0}-\sigma)}{2\sigma}(t-\overline{r}_{i_0}+\sigma)$.
Then it is easy to see that $Z(t)$ has a unique zero point $\overline{s}_{i_0}$ in $(\overline{r}_{i_0-1},\overline{r}_{i_0})$. Using it, we set a $(k+1)$-tuple of functions $(z_1,\cdots,z_{k+1})$ by
\begin{equation}\label{s3.15}
\left\{\begin{array}{ll}
z_{i_0}(t)=\widehat{z}(t),~t\in(\overline{r}_{i_0-1},\overline{s}_{i_0}),\\
z_{i_0+1}(t)=\widehat{z}(t),~t\in(\overline{s}_{i_0},\overline{r}_{i_0+1}),\\
z_{i}(t)=\omega_i(t),~t\in(\overline{r}_{i-1},\overline{r}_{i})~\text{if}~i\neq i_0,i_0+1.
\end{array}\right.
\end{equation}
Furthermore, we obtain by direct computation that
\begin{equation*}
\|z_{i_0}\|_{B_{i_0}'}=\|\omega_{i_0}\|_{B_{i_0}}+o_\sigma(1),~~\|z_{i_0+1}\|_{B_{i_0+1}'}=\|\omega_{i_0+1}\|_{B_{i_0+1}}+o_\sigma(1)
\end{equation*}
and
\begin{equation*}
\displaystyle\int_{B_{i_0}'}|z_{i_0}|^p{\rm d}x=\displaystyle\int_{B_{i_0}}|\omega_{i_0}|^p{\rm d}x+o_\sigma(1),~~\displaystyle\int_{B_{i_0}'}|z_{i_0}|^p{\rm d}x=\displaystyle\int_{B_{i_0}}|\omega_{i_0}|^p{\rm d}x+o_\sigma(1)
\end{equation*}
as $\sigma\rightarrow0^+$. Thus, similarly as the Step 2 in the proof of Lemma \ref{l2.1}, there is a unique $(k+1)$-tuple $(\overline{t}_1,\cdots,\overline{t}_{k+1})\in(\mathbb{R}_{>0})^{k+1}$ such that $(\overline{t}_1z_1,\cdots,\overline{t}_{k+1}z_{k+1})\in N^{\overrightarrow{\widehat{\textbf{r}}}_k,-}$ with $\overrightarrow{\widehat{\textbf{r}}}_k:=(\overline{r}_1,\cdots,\overline{r}_{i_0-1},\overline{s}_{i_0},\overline{r}_{i_0+1},\cdots,\overline{r}_k)$.
Moreover, we have
\begin{equation}\label{s3.16}
\displaystyle\lim_{\sigma\rightarrow0^+}(\overline{t}_1,\cdots,\overline{t}_{k+1})=(1,\cdots,1)
\end{equation}
and
\begin{equation}\label{s3.17}
I_b(W)=E_b^{\overrightarrow{\overline{\textbf{r}}}_k}(\omega_1,\cdots,\omega_{k+1})\leq E_b^{\overrightarrow{\widehat{\textbf{r}}}_k}(\overline{t}_1z_1,\cdots,\overline{t}_{k+1}z_{k+1})=I_b(Z),
\end{equation}
where $W,Z\in H_r^1(\mathbb{R}^3)$ are defined as $W(t):=\displaystyle\sum_{i=1}^{k+1}\omega_i(t)$ and $Z(t):=\displaystyle\sum_{i=1}^{k+1}\overline{t}_iz_i(t)$.
We try to give a contradiction in the following.

Indeed, it holds
\begin{align}
&I_b(Z)-I_b(W)\notag
\\&\leq\left(\displaystyle\int_0^{\overline{r}_{i_0}-\sigma}+\displaystyle\int_{\overline{r}_{i_0}+\sigma}^\infty\right)\left(\frac{1}{2}(Z')^2+\frac{1}{2}V(t)Z^2
-\frac{1}{p}|W|^p-\frac{Z^2-W^2}{2}|W|^{p-2}\right)t^2dt\notag
\\&~~~~-\left(\displaystyle\int_0^{\overline{r}_{i_0}-\sigma}+\displaystyle\int_{\overline{r}_{i_0}+\sigma}^\infty\right)\left(\frac{1}{2}(W')^2+\frac{1}{2}V(t)W^2
-\frac{1}{p}|W|^p\right)t^2dt\notag
\\&~~~~+\displaystyle\int_{\overline{r}_{i_0}-\sigma}^{\overline{r}_{i_0}+\sigma}\left(\frac{1}{2}(Z')^2+\frac{1}{2}V(t)Z^2
-\frac{1}{p}|Z|^p\right)t^2dt\notag
\\&~~~~-\displaystyle\int_{\overline{r}_{i_0}-\sigma}^{\overline{r}_{i_0}+\sigma}\left(\frac{1}{2}(W')^2+\frac{1}{2}V(t)W^2
-\frac{1}{p}|W|^p\right)t^2dt\notag
\\&~~~~+\frac{b}{4}\left(\displaystyle\int_0^\infty(Z')^2t^2dt\right)^2-\frac{b}{4}\left(\displaystyle\int_0^\infty(W')^2t^2dt\right)^2,\label{s3.18}
\end{align}
where we have used the inequality
\begin{equation*}
\frac{c^q}{q}\geq\frac{d^q}{q}+\frac{c^2-d^2}{2}d^{q-2},
\end{equation*}
for any $c,d\geq0$ and $q\geq2$. Noticing the definition of $W$ and $(\omega_1,\cdots,\omega_{k+1})\in N_k^{\overrightarrow{\overline{\textbf{r}}}_k,-}$, we get
\begin{equation}\label{s3.19}
\displaystyle\int_0^\infty((W')^2+V(t)W^2)t^2dt+b\left(\displaystyle\int_0^\infty(W')^2t^2dt\right)^2=\displaystyle\int_0^\infty|W|^pt^2dt.
\end{equation}
Combining (\ref{s3.18}) and (\ref{s3.19}) yields that
\begin{align}
&I_b(Z)-I_b(W)\notag
\\&\leq\left(\displaystyle\int_0^{\overline{r}_{i_0}-\sigma}+\displaystyle\int_{\overline{r}_{i_0}+\sigma}^\infty\right)\left(\frac{1+bA}{2}(Z')^2+\frac{1}{2}V(t)Z^2
-\frac{Z^2|W|^{p-2}}{2}\right)t^2dt\label{s3.20}
\\&~~~~+\displaystyle\int_{\overline{r}_{i_0}-\sigma}^{\overline{r}_{i_0}+\sigma}\left(\frac{1+bA}{2}(Z')^2+\frac{1}{2}V(t)Z^2
-\frac{1}{p}|Z|^p+\frac{1}{p}|W|^p\right)t^2dt\label{s3.21}
\\&~~~~+\frac{b}{4}\left(\displaystyle\int_0^\infty(Z')^2t^2dt\right)^2+\frac{b}{4}\left(\displaystyle\int_0^\infty(W')^2t^2dt\right)^2-\frac{b}{2}\displaystyle\int_0^\infty(W')^2t^2dt\displaystyle\int_0^\infty(Z')^2t^2dt,\label{s3.22}
\end{align}
where we denote
\begin{equation*}
A:=\displaystyle\int_0^\infty(W')^2t^2dt.
\end{equation*}

Let us consider (\ref{s3.20}). Since $W$ satisfies $W(\overline{r}_{i_0})=0$ and
\begin{equation}\label{s3.23}
-(1+bA)(t^2W')'+t^2V(t)W=t^2|W|^{p-2}W,
\end{equation}
for $t\in[\overline{r}_{i_0-1},\overline{r}_{i_0}]$, we have $(t^2W')'(\overline{r}_{i_0})=0$, and so
\begin{equation}\label{s3.24}
W(\overline{r}_{i_0}-\sigma)=-\sigma\omega_-+o(\sigma)~\text{and}~(\overline{r}_{i_0}-\sigma)^2W'(\overline{r}_{i_0}-\sigma)=(\overline{r}_{i_0})^2\omega_-+o(\sigma).
\end{equation}
Noticing (\ref{s3.16}) and (\ref{s3.24}), we obtain by integrating (\ref{s3.23})by parts that
\begin{align}
&~~~~\displaystyle\int_0^{\overline{r}_{i_0}-\sigma}\left(\frac{1+bA}{2}(Z')^2+\frac{1}{2}V(t)Z^2
-\frac{Z^2|W|^{p-2}}{2}\right)t^2dt\notag
\\&=(1+o(1))\displaystyle\int_0^{\overline{r}_{i_0}-\sigma}\left(\frac{1+bA}{2}(W')^2+\frac{1}{2}V(t)W^2
-\frac{|W|^{p}}{2}\right)t^2dt\notag
\\&=\frac{(1+o(1))(1+bA)}{2}W'(\overline{r}_{i_0}-\sigma)W(\overline{r}_{i_0}-\sigma)(\overline{r}_{i_0}-\sigma)^2+o(\sigma)\notag
\\&=-\frac{1+bA}{2}(\omega_-)^2(\overline{r}_{i_0})^2\sigma+o(\sigma).\label{s3.25}
\end{align}
In a similar way, we can get
\begin{align}
&~~~~\displaystyle\int_{\overline{r}_{i_0}+\sigma}^\infty\left(\frac{1+bA}{2}(Z')^2+\frac{1}{2}V(t)Z^2
-\frac{Z^2|W|^{p-2}}{2}\right)t^2dt\notag
\\&=-\frac{1+bA}{2}(\omega_+)^2(\overline{r}_{i_0})^2\sigma+o(\sigma).\label{s3.26}
\end{align}

Next, for (\ref{s3.21}), from (\ref{s3.14}) we can easily check that
\begin{equation}\label{s3.27}
\displaystyle\int_{\overline{r}_{i_0}-\sigma}^{\overline{r}_{i_0}+\sigma}\left(\frac{1}{2}V(t)Z^2
-\frac{|Z|^p-|W|^{p}}{p}\right)t^2dt=o(\sigma)
\end{equation}
and
\begin{equation}\label{s3.28}
\displaystyle\int_{\overline{r}_{i_0}-\sigma}^{\overline{r}_{i_0}+\sigma}\frac{1+bA}{2}(Z')^2t^2dt=\frac{1+bA}{4}(\omega_++\omega_-)^2(\overline{r}_{i_0})^2\sigma+o(\sigma).
\end{equation}

To consider (\ref{s3.22}), we shall prove that
\begin{equation}\label{s3.29}
\overline{t}_i=o(\sigma^\frac{1}{2}),
\end{equation}
for $i=1,\cdots,k+1$. Suppose otherwise, we have
\begin{equation*}
\displaystyle\lim_{\sigma\rightarrow0^+}|\sigma^{-\frac{1}{2}}(\overline{t}_{k_0}-1)|=\displaystyle\max_{1\leq i\leq k+1}\displaystyle\lim_{\sigma\rightarrow0^+}|\sigma^{-\frac{1}{2}}(\overline{t}_i-1)|=B\neq0,
\end{equation*}
for some $k_0\in\{1,\cdots,k+1\}$. Then we obtain
\begin{align}
0&=\displaystyle\lim_{\sigma\rightarrow0^+}\sigma^{-\frac{1}{2}}(F_{k_0}(\overline{t}_1z_1,\cdots,\overline{t}_{k+1}z_{k+1})-F_{k_0}(\omega_1,\cdots,\omega_{k+1}))(\sigma^{-\frac{1}{2}}(\overline{t}_{k_0}-1))^{-1}\notag
\\&\leq2\|\omega_{k_0}\|_{k_0}^2+4b\left(\displaystyle\int_{B_{k_0}^{\overrightarrow{\overline{\textbf{r}}}_k}}|\nabla\omega_{k_0}|^2{\rm d}x\right)^2\notag
\\&~~~~+2b\displaystyle\int_{B_{k_0}^{\overrightarrow{\overline{\textbf{r}}}_k}}|\nabla\omega_{k_0}|^2{\rm d}x\displaystyle\sum_{j\neq k_0}^{k+1}\displaystyle\int_{B_j^{\overrightarrow{\overline{\textbf{r}}}_k}}|\nabla\omega_j|^2{\rm d}x-p\displaystyle\int_{B_{k_0}^{\overrightarrow{\overline{\textbf{r}}}_k}}|\omega_{k_0}|^p{\rm d}x\notag
\\&=(4-p)\displaystyle\int_{B_{k_0}^{\overrightarrow{\overline{\textbf{r}}}_k}}|\omega_{k_0}|^p{\rm d}x-2\|\omega_{k_0}\|_{k_0}^2-2b\displaystyle\int_{B_{k_0}^{\overrightarrow{\overline{\textbf{r}}}_k}}|\nabla\omega_{k_0}|^2{\rm d}x\displaystyle\sum_{j\neq k_0}^{k+1}\displaystyle\int_{B_j^{\overrightarrow{\overline{\textbf{r}}}_k}}|\nabla\omega_j|^2{\rm d}x<0,\notag
\end{align}
where $F_{k_0}$ is defined in (\ref{s2.18}), and we have used the fact $(\omega_1,\cdots,\omega_{k+1})\in N_k^{\overrightarrow{\overline{\textbf{r}}}_k,-}$, $(\overline{t}_1z_1,\cdots,\overline{t}_{k+1}z_{k+1})\in N_k^{\overrightarrow{\widehat{\overline{\textbf{r}}}_k},-}$
and $\overrightarrow{\widehat{\overline{\textbf{r}}}_k}\rightarrow\overrightarrow{\overline{\textbf{r}}}_k$ as $\sigma\rightarrow 0^+$. This is a contradiction and
we obtain (\ref{s3.29}). Hence,
\begin{align}
&~~~~\frac{b}{4}\left(\displaystyle\int_0^\infty(Z')^2t^2dt\right)^2+\frac{b}{4}\left(\displaystyle\int_0^\infty(W')^2t^2dt\right)^2-\frac{b}{2}\displaystyle\int_0^\infty(W')^2t^2dt\displaystyle\int_0^\infty(Z')^2t^2dt\notag
\\&=\frac{b}{4}\left(\displaystyle\int_0^\infty(W')^2t^2dt-\displaystyle\int_0^\infty(Z')^2t^2dt\right)^2\notag
\\&=\frac{b}{4}[o(\sigma^\frac{1}{2})\left(\displaystyle\int_0^{\overline{r}_{i_0}-\sigma}+\displaystyle\int_{\overline{r}_{i_0}+\sigma}^\infty\right)(W')^2t^2dt\notag
\\&~~~~+\left(\frac{(\omega_++\omega_-)^2}{2}-((\omega_+)^2+(\omega_-)^2)\right)(\overline{r}_{i_0})^2\sigma]^2=o(\sigma).\label{s3.30}
\end{align}

Combining (\ref{s3.25})-(\ref{s3.28}) and (\ref{s3.30}), we get
\begin{equation*}
I_b(Z)-I_b(W)\leq-\frac{a+bA}{4}(\omega_+-\omega_-)(\overline{r}_{i_0})^2\sigma+o(\sigma),
\end{equation*}
which means that $I_b(Z)-I_b(W)<0$ if we take $\sigma>0$ small enough. This contradicts (\ref{s3.17}). Then we complete the proof.
\end{proof}

\section{\normalsize{Energy comparison and some properties of the solutions}}

\quad\quad
From Theorem 1.1, we see that (\ref{s1.1}) admits a radial solution $u_k$ which changes exactly $k$-times for any integer $k\geq0$. In this Section, we shall show that $E_b(u_k)$
is strictly increasing in $k$. In particular, we prove that $E_b(u_k)>(k+1)E_b(u_0)$. Now, let us give the proof of Theorem 1.2.
\begin{proof}[Proof of Theorem 1.2]
It follows from Lemma \ref{l3.1} that there is a vector $\overrightarrow{\overline{\textbf{r}}}_k=(\overline{r}_1,\cdots,\overline{r}_k)\in\Lambda_k$ such that
\begin{equation*}
\varphi_k(\overrightarrow{\overline{\textbf{r}}}_k)=\displaystyle\inf_{\overrightarrow{\textbf{r}}_k\in\Lambda_k}\varphi_k(\overrightarrow{\textbf{r}}_k),
\end{equation*}
where $\varphi_k$ is defined in (\ref{s3.1}) and the subscript $k$ is used to emphasize the dependence on $k$. Moreover, Lemma \ref{l2.4} shows that a vector
$(\omega_1^{\overrightarrow{\overline{\textbf{r}}}_k},\cdots,\omega_{k+1}^{\overrightarrow{\overline{\textbf{r}}}_k})\in N_k^{\overrightarrow{\overline{\textbf{r}}}_k,-}$
satisfies the following system
\begin{equation*}
\left\{\begin{array}{ll}
-\left(a+b\displaystyle\sum_{j=1}^{k+1}\displaystyle\int_{B^{\overrightarrow{\overline{\textbf{r}}}_k}_j}|\nabla u_j|^2{\rm d}x\right)\Delta u_i+V(|x|)u_i=|u_i|^{p-2}u_i,~ x\in B^{\overrightarrow{\overline{\textbf{r}}}_k}_i,\\
u_i=0,~x\not\in B^{\overrightarrow{\overline{\textbf{r}}}_k}_i.
\end{array}\right.
\end{equation*}
Furthermore, by Theorem 1.1,  we see that
\begin{equation*}
u_k:=\omega_1^{\overrightarrow{\overline{\textbf{r}}}_k}+\cdots+\omega_{k+1}^{\overrightarrow{\overline{\textbf{r}}}_k},
\end{equation*}
is a solution of (\ref{s1.1}) which changes sign exactly $k$-times. Denote also
\begin{equation*}
u_{k+1}:=\omega_1^{\overrightarrow{\overline{\textbf{r}}}_{k+1}}+\cdots+\omega_{k+2}^{\overrightarrow{\overline{\textbf{r}}}_{k+1}},
\end{equation*}
as a solution of (\ref{s1.1}) which changes sign exactly $k+1$-times, and $\overrightarrow{\overline{\textbf{r}}}_{k+1}=(\overline{r}_1,\cdots,\overline{r}_{k+1})\in\Lambda_{k+1}$ is obtained from
Lemma \ref{l3.1}.

Let
\begin{equation*}
\overrightarrow{\widetilde{\textbf{r}}}_{k}=(\overline{r}_2,\overline{r}_3,\cdots,\overline{r}_{k+1}).
\end{equation*}
Then, ones deduce from Lemma \ref{l2.4} that there is a minimizer $(\omega_1^{\overrightarrow{\widetilde{\textbf{r}}}_{k}},\cdots,\omega_{k+1}^{\overrightarrow{\widetilde{\textbf{r}}}_{k}})$
of $E_b^{\overrightarrow{\widetilde{\textbf{r}}}_{k}}|_{N_k^{\overrightarrow{\widetilde{\textbf{r}}}_{k},-}}$ such that
\begin{equation}\label{s4.1}
E_b^{\overrightarrow{\widetilde{\textbf{r}}}_{k}}(\omega_1^{\overrightarrow{\widetilde{\textbf{r}}}_{k}},\cdots,\omega_{k+1}^{\overrightarrow{\widetilde{\textbf{r}}}_{k}})=\displaystyle\inf_{(u_1,\cdots,u_{k+1})\in N_k^{\overrightarrow{\widetilde{\textbf{r}}}_{k},-}}E_b^{\overrightarrow{\widetilde{\textbf{r}}}_{k}}(u_1,\cdots,u_{k+1}).
\end{equation}
Since $B_1^{\overrightarrow{\overline{\textbf{r}}}_{k+1}}\subset B_1^{\overrightarrow{\widetilde{\textbf{r}}}_{k}}$, we have $H_1^{\overrightarrow{\overline{\textbf{r}}}_{k+1}}(B_1^{\overrightarrow{\overline{\textbf{r}}}_{k+1}})\subset H_1^{\overrightarrow{\widetilde{\textbf{r}}}_{k}}(B_1^{\overrightarrow{\widetilde{\textbf{r}}}_{k}})$.
Thus, $(\omega_1^{\overrightarrow{\overline{\textbf{r}}}_{k+1}},\omega_3^{\overrightarrow{\overline{\textbf{r}}}_{k+1}},\cdots,\omega_{k+2}^{\overrightarrow{\overline{\textbf{r}}}_{k+1}})$
is an element of $\mathcal{\textbf{H}}_k^{\overrightarrow{\widetilde{\textbf{r}}}_{k}}$. Noticing that $(\omega_1^{\overrightarrow{\overline{\textbf{r}}}_{k+1}},\omega_2^{\overrightarrow{\overline{\textbf{r}}}_{k+1}},\cdots,\omega_{k+2}^{\overrightarrow{\overline{\textbf{r}}}_{k+1}})\in N_{k+1}^{\overrightarrow{\overline{\textbf{r}}}_{k+1}}$,
we have
\begin{align}
&\|\omega_i^{\overrightarrow{\overline{\textbf{r}}}_{k+1}}\|_i^2+b\left(\displaystyle\int_{B_i^{\overrightarrow{\overline{\textbf{r}}}_{k+1}}}|\nabla\omega_i^{\overrightarrow{\overline{\textbf{r}}}_{k+1}}|^2{\rm d}x\right)^2\notag
\\&~~~~+b\displaystyle\int_{B_i^{\overrightarrow{\overline{\textbf{r}}}_{k+1}}}|\nabla\omega_i^{\overrightarrow{\overline{\textbf{r}}}_{k+1}}|^2{\rm d}x\displaystyle\sum_{j\neq i,j\neq2}^{k+1}\displaystyle\int_{B_j^{\overrightarrow{\overline{\textbf{r}}}_{k+1}}}|\nabla\omega_j^{\overrightarrow{\overline{\textbf{r}}}_{k+1}}|^2{\rm d}x
-\displaystyle\int_{B_i^{\overrightarrow{\overline{\textbf{r}}}_{k+1}}}|\omega_i^{\overrightarrow{\overline{\textbf{r}}}_{k+1}}|^p{\rm d}x<0,\label{s4.2}
\end{align}
and
\begin{equation}\label{s4.3}
(4-p)\displaystyle\int_{B_i^{\overrightarrow{\overline{\textbf{r}}}_{k+1}}}|\omega_i^{\overrightarrow{\overline{\textbf{r}}}_{k+1}}|^p{\rm d}x<2\|\omega_i^{\overrightarrow{\overline{\textbf{r}}}_{k+1}}\|_i^2
\end{equation}
for $i=1,3,\cdots,k+2$. Consequently, with the help of Lemma \ref{l2.3}, we obtain that
there exists a unique $(k+1)$-tuple $(\widehat{t}_1,\widehat{t}_3,\cdots,\widehat{t}_{k+1})\neq(1,1,\cdots,1)$ of positive numbers such that
\begin{equation*}
(\widehat{t}_1\omega_1^{\overrightarrow{\overline{\textbf{r}}}_{k+1}},\widehat{t}_3\omega_3^{\overrightarrow{\overline{\textbf{r}}}_{k+1}},\cdots,\widehat{t}_{k+2}\omega_{k+2}^{\overrightarrow{\overline{\textbf{r}}}_{k+1}})\in N_k^{\overrightarrow{\widetilde{\textbf{r}}}_{k},-}
\end{equation*}
 with $\widehat{t}_i\leq1$ for $i=1,3,\cdots,k+2$. Moreover, we have
 \begin{equation}\label{s4.4}
E_b^{\overrightarrow{\widetilde{\textbf{r}}}_{k}}(\omega_1^{\overrightarrow{\overline{\textbf{r}}}_{k+1}},\omega_3^{\overrightarrow{\overline{\textbf{r}}}_{k+1}},\cdots,\omega_{k+2}^{\overrightarrow{\overline{\textbf{r}}}_{k+1}})<E_b^{\overrightarrow{\widetilde{\textbf{r}}}_{k}}(\widehat{t}_1\omega_1^{\overrightarrow{\overline{\textbf{r}}}_{k+1}},\widehat{t}_3\omega_3^{\overrightarrow{\overline{\textbf{r}}}_{k+1}},\cdots,\widehat{t}_{k+2}\omega_1^{\overrightarrow{\overline{\textbf{r}}}_{k+1}}).
\end{equation}
Similarly, we have
\begin{align}
&~~~~E_b^{\overrightarrow{\widetilde{\textbf{r}}}_{k}}(\widehat{t}_1\omega_1^{\overrightarrow{\overline{\textbf{r}}}_{k+1}},\widehat{t}_3\omega_3^{\overrightarrow{\overline{\textbf{r}}}_{k+1}},\cdots,\widehat{t}_{k+2}\omega_{k+2}^{\overrightarrow{\overline{\textbf{r}}}_{k+1}})\notag
\\&=E_b^{\overrightarrow{\overline{\textbf{r}}}_{k+1}}(\widehat{t}_1\omega_1^{\overrightarrow{\overline{\textbf{r}}}_{k+1}},0,\widehat{t}_3\omega_3^{\overrightarrow{\overline{\textbf{r}}}_{k+1}},\cdots,\widehat{t}_{k+2}\omega_{k+2}^{\overrightarrow{\overline{\textbf{r}}}_{k+1}})\notag
\\&<E_b^{\overrightarrow{\overline{\textbf{r}}}_{k+1}}(\omega_1^{\overrightarrow{\overline{\textbf{r}}}_{k+1}},\omega_2^{\overrightarrow{\overline{\textbf{r}}}_{k+1}},\omega_3^{\overrightarrow{\overline{\textbf{r}}}_{k+1}},\cdots,\omega_{k+2}^{\overrightarrow{\overline{\textbf{r}}}_{k+1}}).\label{s4.5}
\end{align}
On the other hand, based on the definition of $\overrightarrow{\overline{\textbf{r}}}_k$, ones have
\begin{equation}\label{s4.6}
E_b^{\overrightarrow{\overline{\textbf{r}}_k}}(\omega_1^{\overrightarrow{\overline{\textbf{r}}_k}},\cdots,\omega_{k+1}^{\overrightarrow{\overline{\textbf{r}}_k}})\leq
E_b^{\overrightarrow{\widetilde{\textbf{r}}_k}}(\omega_1^{\overrightarrow{\widetilde{\textbf{r}}_k}},\cdots,\omega_{k+1}^{\overrightarrow{\widetilde{\textbf{r}}_k}}),
\end{equation}
which together with (\ref{s4.4})-(\ref{s4.5}) show that
\begin{align}
I_b(u_k)&=E_b^{\overrightarrow{\overline{\textbf{r}}_k}}(\omega_1^{\overrightarrow{\overline{\textbf{r}}_k}},\cdots,\omega_{k+1}^{\overrightarrow{\overline{\textbf{r}}_k}})\notag
\\&<E_b^{\overrightarrow{\overline{\textbf{r}}}_{k+1}}(\omega_1^{\overrightarrow{\overline{\textbf{r}}}_{k+1}},\omega_2^{\overrightarrow{\overline{\textbf{r}}}_{k+1}},\omega_3^{\overrightarrow{\overline{\textbf{r}}}_{k+1}},\cdots,\omega_{k+2}^{\overrightarrow{\overline{\textbf{r}}}_{k+1}})=I_b(u_{k+1}).\notag
\end{align}

Since $u_k=\omega_1^{\overrightarrow{\overline{\textbf{r}}_k}}+\cdots+\omega_{k+1}^{\overrightarrow{\overline{\textbf{r}}_k}}$ is a solution of (\ref{s1.1}) changing sign exactly $k$
times, it follows from Lemma \ref{l2.4} that $(-1)^{i+1}\omega_i^{\overrightarrow{\overline{\textbf{r}}_k}}$ is positive for each $i=1,\cdots,k+1$.
We denote $\widehat{\omega}_i:=(-1)^{i+1}\omega_i^{\overrightarrow{\overline{\textbf{r}}_k}}$ for simplicity. Clearly, we see
\begin{equation*}
\|\widehat{\omega}_i\|_i^2+b\left(\displaystyle\int_{B_i^{\overrightarrow{\overline{\textbf{r}}}_{k+1}}}|\nabla\widehat{\omega}_i|^2{\rm d}x\right)^2-\displaystyle\int_{B_i^{\overrightarrow{\overline{\textbf{r}}}_{k+1}}}|\widehat{\omega}|^p{\rm d}x<0
\end{equation*}
and
\begin{equation*}
(4-p)\displaystyle\int_{B_i^{\overrightarrow{\overline{\textbf{r}}}_{k+1}}}|\widehat{\omega}_i|^p{\rm d}x<2\|\widehat{\omega}_i\|_i^2
\end{equation*}
for $i=1,\cdots,k+1$. Thus it follows from Lemma \ref{l2.3} that there exists a unique $0<\widehat{\widehat{t}}_i<1$ such that
\begin{equation*}
\widehat{\widehat{t}}_i\widehat{\omega}_i\in N_0^{{\overrightarrow{\overline{\textbf{r}}}_{k+1}},-}=N_0^-,
\end{equation*}
where  $N_0^-$ is defined by
\begin{equation*}
N_0^-:=\{u\in\mathcal{H}\backslash\{0\}|\langle(I_b)'(u),u\rangle=0,~\langle(J_b)'(u),u\rangle<0\}
\end{equation*}
and
\begin{equation*}
J_b(u):=\langle(I_b)'(u),u\rangle=\|u\|_{\mathcal{H}}^2+b\left(\displaystyle\int_{\mathbb{R}^3}|\nabla u|^2{\rm d}x\right)^2-\displaystyle\int_{\mathbb{R}^3}|u|^p{\rm d}x.
\end{equation*}
Moreover, it holds
\begin{equation*}
I_b(u_0)\leq I_b(\widehat{\widehat{t}}_i\widehat{\omega}_i),
\end{equation*}
for $i=1,\cdots,k+1$. Consequently, by a direct computation,
\begin{align}
&~~~~(k+1)I_b(u_0)\leq\displaystyle \sum_{i=1}^{k+1}I_b(\widehat{\widehat{t}}_i\widehat{\omega}_i)=\displaystyle \sum_{i=1}^{k+1}\left(I_b(\widehat{\widehat{t}}_i\widehat{\omega}_i)-\frac{1}{4}J_b(\widehat{\widehat{t}}_i\widehat{\omega}_i)\right)\notag
\\&=\frac{1}{4}\displaystyle \sum_{i=1}^{k+1}(\widehat{\widehat{t}}_i)^2\|\widehat{\omega}_i\|_i^2-\frac{4-p}{4p}\displaystyle \sum_{i=1}^{k+1}(\widehat{\widehat{t}}_i)^p\displaystyle\int_{B_i^{\overrightarrow{\overline{\textbf{r}}}_{k}}}|\widehat{\omega}_i|^p{\rm d}x\notag
\\&<\frac{1}{4}\displaystyle \sum_{i=1}^{k+1}(\widehat{\widehat{t}}_i)^2\left(\|\widehat{\omega}_i\|_i^2+b\displaystyle\int_{B_i^{\overrightarrow{\overline{\textbf{r}}}_{k}}}|\nabla\widehat{\omega}_i|^2{\rm d}x\displaystyle\sum_{j\neq i}^{k+1}(\widehat{\widehat{t}}_j)^2\displaystyle\int_{B_j^{\overrightarrow{\overline{\textbf{r}}}_{k}}}|\nabla\widehat{\omega}_j|^2{\rm d}x\right)\notag
\\&~~~~-\frac{4-p}{4p}\displaystyle \sum_{i=1}^{k+1}(\widehat{\widehat{t}}_i)^p\displaystyle\int_{B_i^{\overrightarrow{\overline{\textbf{r}}}_{k}}}|\widehat{\omega}_i|^p{\rm d}x\notag
\\&<\frac{1}{4}\displaystyle \sum_{i=1}^{k+1}\left(\|\widehat{\omega}_i\|_i^2+b\displaystyle\int_{B_i^{\overrightarrow{\overline{\textbf{r}}}_{k}}}|\nabla\widehat{\omega}_i|^2{\rm d}x\displaystyle\sum_{j\neq i}^{k+1}\displaystyle\int_{B_j^{\overrightarrow{\overline{\textbf{r}}}_{k}}}|\nabla\widehat{\omega}_j|^2{\rm d}x\right)-\frac{4-p}{4p}\displaystyle \sum_{i=1}^{k+1}\displaystyle\int_{B_i^{\overrightarrow{\overline{\textbf{r}}}_{k}}}|\widehat{\omega}_i|^p{\rm d}x\notag
\\&=E_b^{\overrightarrow{\overline{\textbf{r}}_k}}(\omega_1^{\overrightarrow{\overline{\textbf{r}}_k}},\cdots,\omega_{k+1}^{\overrightarrow{\overline{\textbf{r}}_k}})=I_b(u_k),\notag
\end{align}
where we have used Lemma \ref{l2.3}.

To complete the proof, we need to show that $u_0$ is a ground state radial solution of (\ref{s1.1}). Noticing that $u_0$ is a minimizer of  $N_0^-$,  we only need to show that all the least energy radial solutions of
(\ref{s1.1}) are in $N_0^-$. Actually, let $v_0$ be a least energy radial solutions of
(\ref{s1.1}). Then, we deduce from (\ref{s1.2}) that
\begin{equation}\label{s4.7}
\|v_0\|_{\mathcal{H}}^2+b\left(\displaystyle\int_{\mathbb{R}^3}|\nabla v_0|^2{\rm d}x\right)^2=\displaystyle\int_{\mathbb{R}^3}|\nabla v_0|^p{\rm d}x.
\end{equation}
Moreover, since $\langle\nabla V(x),x\rangle\leq0$ for any $x\in\mathbb{R}^3$, we find from \cite{28} that $v_0$ also satisfies the following Pohoz\v{a}ev identity
\begin{align}\label{s4.8}
&\frac{1}{2}\displaystyle\int_{\mathbb{R}^3}|\nabla v_0|^2{\rm d}x+\frac{3}{2}\displaystyle\int_{\mathbb{R}^3}V(x)|v_0|^2{\rm d}x+\frac{1}{2}\displaystyle\int_{\mathbb{R}^3}\langle \nabla V(x),x\rangle|v_0|^2{\rm d}x\notag
\\&~~~~+\frac{b}{2}\left(\displaystyle\int_{\mathbb{R}^3}|\nabla v_0|^2{\rm d}x\right)^2-\frac{3}{p}\displaystyle\int_{\mathbb{R}^3}|\nabla v_0|^p{\rm d}x=0.
\end{align}
Combining (\ref{s4.7}) and  (\ref{s4.8}) yield that
\begin{equation}\label{s4.9}
\displaystyle\int_{\mathbb{R}^3}V(x)|v_0|^2{\rm d}x+\frac{1}{2}\displaystyle\int_{\mathbb{R}^3}\langle \nabla V(x),x\rangle|v_0|^2{\rm d}x=\frac{6-p}{2p}\displaystyle\int_{\mathbb{R}^3}|\nabla v_0|^p{\rm d}x,
\end{equation}
which together with the condition $\langle\nabla V(x),x\rangle\leq0$ and the assumption $3\leq p<4$ show that
\begin{equation}\label{s4.10}
\displaystyle\int_{\mathbb{R}^3}|\nabla v_0|^p{\rm d}x<\frac{2}{4-p}\|v_0\|_{\mathcal{H}}^2.
\end{equation}
Therefore, by (\ref{s4.7}) and  (\ref{s4.10}), we see $ v_0\in N_0^-$.
Then we complete the proof.
\end{proof}

Finally, we will prove the convergence property of $u_k^b$ as $b\rightarrow0^+$.
\begin{proof}
For any $b\in(0,b_*)$, it follows from Theorems 1.1 and 1.2 that there is a radial solution  $u_k^b\in\mathcal{H}$ of (\ref{s1.1}) which changes sign exactly $k$ times.
Let $b_n\rightarrow0^+$ be any sequence as $n\rightarrow\infty$, since $\overrightarrow{\overline{\textbf{r}}}_{k}=(\overline{r}_1,\cdots,\overline{r}_k)\in\Lambda_k$ obtained by
Lemma \ref{l3.1} does not depend on $b$, we know that a family of annuli $\{B_i^{\overrightarrow{\overline{\textbf{r}}}_{k}}\}_{i=1}^{k+1}$ are fixed. Moreover, clearly, that
\begin{equation*}
\alpha_k^{\overrightarrow{\overline{\textbf{r}}}_{k},-}:=\displaystyle\inf_{(\omega_1,\cdots,\omega_{k+1})\in N_k^{\overrightarrow{\overline{\textbf{r}}}_{k},-}}E_b^{\overrightarrow{\overline{\textbf{r}}}_{k}}(\omega_1,\cdots,\omega_{k+1})
\end{equation*}
is decreasing as $b\searrow0^+$. Thus, there exists $C_0>0$ such that
\begin{align}
C_0&>E_{b_n}^{\overrightarrow{\overline{\textbf{r}}}_{k}}(\omega^{\overrightarrow{\overline{\textbf{r}}}_{k}}_{1,b_n},\cdots,\omega^{\overrightarrow{\overline{\textbf{r}}}_{k}}_{k+1,b_n})\notag
\\&=E_{b_n}^{\overrightarrow{\overline{\textbf{r}}}_{k}}(\omega^{\overrightarrow{\overline{\textbf{r}}}_{k}}_{1,b_n},\cdots,\omega^{\overrightarrow{\overline{\textbf{r}}}_{k}}_{k+1,b_n})-\frac{1}{4}\displaystyle \sum_{i=1}^{k+1}\langle\partial_iE_{b_n}^{\overrightarrow{\overline{\textbf{r}}}_{k}}(\omega^{\overrightarrow{\overline{\textbf{r}}}_{k}}_{1,b_n},\cdots,\omega^{\overrightarrow{\overline{\textbf{r}}}_{k}}_{k+1,b_n}),\omega^{\overrightarrow{\overline{\textbf{r}}}_{k}}_{i,b_n}\rangle\notag
\\&>\frac{1}{4}\displaystyle \sum_{i=1}^{k+1}\|\omega^{\overrightarrow{\overline{\textbf{r}}}_{k}}_{i,b_n}\|_i^2-\left(\frac{1}{p}-\frac{1}{4}\right)\displaystyle \sum_{i=1}^{k+1}\displaystyle\int_{B_i^{\overrightarrow{\overline{\textbf{r}}}_{k}}}|\omega^{\overrightarrow{\overline{\textbf{r}}}_{k}}_{i,b_n}|^p{\rm d}x\notag
\\&>\frac{p-2}{4p}\displaystyle \sum_{i=1}^{k+1}\|\omega^{\overrightarrow{\overline{\textbf{r}}}_{k}}_{i,b_n}\|_i^2=\frac{p-2}{4p}\|u_k^{b_n}\|_{\mathcal{H}}^2,\notag
\end{align}
where we have used the fact $(\omega^{\overrightarrow{\overline{\textbf{r}}}_{k}}_{1,b_n},\cdots,\omega^{\overrightarrow{\overline{\textbf{r}}}_{k}}_{k+1,b_n})\in N_k^{\overrightarrow{\overline{\textbf{r}}}_{k},-}$. Then we obtain that $\{u_k^{b_n}\}$ is bounded in $\mathcal{H}$. Thus, there exists a subsequence of
$\{u_k^{b_n}\}$ we still denote by it for convenience, such that $u_k^{b_n}\rightharpoonup u_k^0\neq0$ weakly in $\mathcal{H}$. Furthermore, standard argument shows that $u_k^0$
is a weak solution of (\ref{s1.3}). By the compactness of the embedding $\mathcal{H}\rightharpoonup L^q(\mathbb{R}^3)$ for $2<q<6$, we have
\begin{align}
&~~~~\|u_k^{b_n}-u_k^0\|_{\mathcal{H}}^2\notag
\\&=\langle I'_{b_n}(u_k^{b_n})-I_0'{u_k^0},u_k^{b_n}-u_k^0\rangle-b_n\displaystyle\int_{\mathbb{R}^3}|\nabla u_k^{b_n}|^2{\rm d}x\displaystyle\int_{\mathbb{R}^3}\nabla u_k^{b_n}(\nabla u_k^{b_n}-\nabla u_k^0){\rm d}x\notag
\\&~~~~+\displaystyle\int_{\mathbb{R}^3}|u_k^{b_n}|^{p-2}u_k^{b_n}(u_k^{b_n}-u_k^0){\rm d}x-\displaystyle\int_{\mathbb{R}^3}|u_k^0|^{p-2}u_k^0(u_k^{b_n}-u_k^0){\rm d}x\rightarrow0,\notag
\end{align}
as $n\rightarrow\infty$. Then we get $u_k^{b_n}\rightarrow u_k^0$ strongly in $\mathcal{H}$. Moreover, it follows from $(\omega^{\overrightarrow{\overline{\textbf{r}}}_{k}}_{1,b_n},\cdots,\omega^{\overrightarrow{\overline{\textbf{r}}}_{k}}_{k+1,b_n})\in N_k^{\overrightarrow{\overline{\textbf{r}}}_{k},-}$. Then we obtain that $\{u_k^{b_n}\}$
that
\begin{equation*}
\|\omega^{\overrightarrow{\overline{\textbf{r}}}_{k}}_{i,b_n}\|_i^2\leq\displaystyle\int_{B_i^{\overrightarrow{\overline{\textbf{r}}}_{k}}}|\omega^{\overrightarrow{\overline{\textbf{r}}}_{k}}_{i,b_n}|^p{\rm d}x\leq S_p^\frac{p}{2}\|\omega^{\overrightarrow{\overline{\textbf{r}}}_{k}}_{1,b_n}\|_i^p,
\end{equation*}
and so
\begin{equation*}
\|\omega^{\overrightarrow{\overline{\textbf{r}}}_{k}}_{i,b_n}\|_i^{p-2}\geq C_i>0,
\end{equation*}
for $i=1,\cdots,k+1$. Thus,
\begin{equation*}
\|\omega^{\overrightarrow{\overline{\textbf{r}}}_{k}}_{i,0}\|_i^{p-2}\geq C_i>0,
\end{equation*}
for $i=1,\cdots,k+1$, i.e., $u_k^0$ changes sign exactly $k$ times.

Let $v_k$ be a least energy radial solution of (\ref{s1.3}), and $v_k=v_{k,1}+\cdots+v_{k,k+1}$ with $v_{k,i}$ is supported on only one annulus $B_i^{\overrightarrow{\overline{\textbf{r}}}_{k}}$ and
vanishes at the complement of it for $i=1,\cdots,k+1$. Then we have
\begin{equation}\label{s4.11}
I_0(v_k)=\left(\frac{1}{2}-\frac{1}{p}\right)\|v_k\|_{\mathcal{H}}^2=\left(\frac{1}{2}-\frac{1}{p}\right)\displaystyle\sum_{i=1}^{k+1}\|v_{k,i}\|_i^2=\alpha_{0,k},
\end{equation}
where
\begin{equation*}
\alpha_{0,k}:=\displaystyle\inf_{(u_1,\ldots,u_{k+1})\in M_k}I_0(u_1+\cdots+u_{k+1})
\end{equation*}
and
\begin{equation}\label{s4.12}
M_k:=\left\{(u_1,\ldots,u_{k+1})\in\mathcal{H}_k|\|v_{k,i}\|_i^2=\displaystyle\int_{B_i}|u_i|^p{\rm d}x\right\}.
\end{equation}
The existence of $v_k$ can be obtained by a similar proof of Theorem 1.1.
Furthermore, (\ref{s4.11}) and (\ref{s4.12}) yield that
\begin{equation*}
S_p^\frac{p}{p-2}\leq\|v_{k,i}\|_i^2=\displaystyle\int_{B_i}|v_{k,i}|^p{\rm d}x=\frac{2p}{p-2}\alpha_{0,k},
\end{equation*}
for $i=1,\cdots,k+1$.
Then, similarly as the proof of Lemma \ref{l2.1},
there exists a unique $(k+1)$-tuple $(\widehat{t}_1(b_n),\cdots,\widehat{t}_{k+1}(b_n))$ of positive numbers such that
\begin{equation*}
(\widehat{t}_1(b_n)v_{k,1},\cdots,\widehat{t}_{k+1}(b_n)v_{k,k+1})\in N_k^{\overrightarrow{\overline{\textbf{r}}}_{k},-},
\end{equation*}
for $n$ large enough. Moreover, ones deduce that
\begin{equation*}
(\widehat{t}_1(b_n),\cdots,\widehat{t}_{k+1}(b_n))\rightarrow(1,\cdots,1),
\end{equation*}
as $n\rightarrow\infty$. Consequently,
\begin{align}
I_0(v_k)&\leq I_0(u_k^0)=\displaystyle\lim_{n\rightarrow\infty}I_{b_n}(u_k^{b_n})=\displaystyle\lim_{n\rightarrow\infty}E_{b_n}^{\overrightarrow{\overline{\textbf{r}}}_{k}}(\omega^{\overrightarrow{\overline{\textbf{r}}}_{k}}_{1,b_n},\cdots,\omega^{\overrightarrow{\overline{\textbf{r}}}_{k}}_{k+1,b_n})\notag
\\&\leq\displaystyle\lim_{n\rightarrow\infty}E_{b_n}^{\overrightarrow{\overline{\textbf{r}}}_{k}}(\widehat{t}_1(b_n)v_{k,1},\cdots,\widehat{t}_{k+1}(b_n)v_{k,k+1})\notag
\\&=E_0^{\overrightarrow{\overline{\textbf{r}}}_{k}}(v_{k,1},\cdots,v_{k,k+1})=I_0(v_k),\notag
\end{align}
and so $u_k^0$ is a least energy radial solution of (\ref{s1.3}) which changes signs exactly $k$ times. The proof of Theorem 1.3 is completed.
\end{proof}

%\section{ Acknowledgement}
%This work was supported by the National Natural Science Foundation of China (Grant No.12101599). Moreover, all the authors in this work do not have any conflicts of interest.
%\section{ Data Availability}
%The data that support the findings of this study are available from the corresponding author
%upon reasonable request.

\medskip

\noindent{\bf Data availability.} We declare that the manuscript has no associated data.
\vskip3pt
\noindent{\bf Conflict of interest.} The authors have no conflict of interest to declare for this article.

\medbreak

\noindent{\bf Acknowledgements}\,\, Marco Squassina is supported by Gruppo Nazionale per l'Analisi Ma\-te\-ma\-ti\-ca, la Probabilit\`a e le loro Applicazioni. H. Fan is supported by National Natural Science Foundation of China (12101599). J. Zhang is supported by National Natural Science Foundation of China (12371109).

\end{document}